\documentclass[11pt]{article}
\usepackage[left=3cm,top=2.5cm,right=3cm]{geometry}             
\usepackage[parfill]{parskip}    

\usepackage{graphicx,caption,subcaption} 

\usepackage{url}

\usepackage{enumerate}                            
\usepackage{appendix}                               

\usepackage{amssymb,amsmath,amsthm} 

\usepackage{hyperref}

\hypersetup{
    colorlinks=true, 
    linktoc=all,     
    linkcolor=blue,  
    urlcolor=blue
}

\newcommand{\RR}{\mathbb{R}}

\newcommand{\CC}{\mathbb{C}}
\newcommand{\ZZ}{\mathbb{Z}}

\newcommand{\SSS}{\mathbb{S}}

\newcommand{\HH}{\mathcal{Q}}

\newcommand{\QQ}{\mathfrak{q}}
\newcommand{\AAA}{\mathfrak{a}}

\newcommand{\Res}{\mathcal{R}}

\newcommand{\haus}{\mathcal H}

\newcommand{\eps}{\epsilon}                      

\newcommand{\tu}{\tilde{u}}
\newcommand{\tf}{\tilde{f}}
\newcommand{\tg}{\tilde{g}}    
\newcommand{\tsig}{\tilde{\sigma}} 
 
\newcommand{\tM}{\widetilde{M}}

\newcommand{\tV}{\widetilde{V}}
\newcommand{\tT}{\widetilde{T}}

\newcommand{\tr}{\tilde{r}} 
\newcommand{\tth}{\theta} 
\newcommand{\ts}{s} 
\newcommand{\tx}{\tilde{x}} 

\usepackage{color} 
\usepackage[normalem]{ulem}                        

\numberwithin{equation}{section}

\makeatletter
\newtheorem*{rep@theorem}{\rep@title}
\newcommand{\newreptheorem}[2]{%
\newenvironment{rep#1}[1]{%
 \def\rep@title{#2 \ref{##1}}%
 \begin{rep@theorem}}%
 {\end{rep@theorem}}}
\makeatother

\newtheorem{theorem}{Theorem}[section]
\newreptheorem{theorem}{Theorem}

\newtheorem{lemma}[theorem]{Lemma}
\newreptheorem{lemma}{Lemma}

\newtheorem{prop}[theorem]{Proposition}
\newreptheorem{prop}{Proposition}

\newtheorem{corollary}[theorem]{Corollary}
\newreptheorem{corollary}{Corollary}

\newtheorem{defn}[theorem]{Definition}
\newreptheorem{defn}{Definition}

\theoremstyle{definition}
\newtheorem{remark}[theorem]{Remark}

\title{Transformation Optics for the Modelling of Waves in a Universe with Nontrivial Topology}
\author{Tracey Balehowsky\thanks{\href{tracey.balehowsky@helsinki.fi}{tracey.balehowsky@helsinki.fi} } \\ University of Helsinki\and Matti Lassas\thanks{\href{matti.lassas@helsinki.fi}{matti.lassas@helsinki.fi}}\\ University of Helsinki \and Pekka Pankka\thanks{\href{pekka.pankka@helsinki.fi}{pekka.pankka@helsinki.fi}}\\ University of Helsinki \and Ville Sirvi\"{o}\thanks{\href{ville.sirvio@helsinki.fi}{ville.sirvio@helsinki.fi}}\\ University of Helsinki}
\date{\today}                                         
\begin{document}

\maketitle

\begin{abstract}
We consider how transformation optics and invisibility cloaking can be used to construct models in subsets $\mathbb R^3$ with a varying metric, where the time-harmonic waves for a given angular wavenumber $k$, are equivalent to the waves in some closed orientable manifold.
The obtained models could in principle be physically implemented using a device built from metamaterials. In particular the measurements in the metamaterial device given by the Helmholtz source-to-solution operator are equivalent to Helmholtz source-to-solution measurements in a universe given by $(\RR_+\times M, -dt^2 +g)$, where $(M,g)$ is a closed, orientable, $C^\infty$-smooth, 3-dimensional Riemannian manifold. Thus the obtained construction could be used to simulate cosmological models using metamaterial devices.

\end{abstract}

\tableofcontents

%
%


%
%

\section{Introduction}

The study of cosmic topology seeks to answer a fundamental question: what is the shape of the universe? By shape, cosmologists mean to identify both the topological structure of the universe and its Lorentzian spacetime structure. It is assumed that the universe is given by a $C^\infty$-smooth differentiable 4-manifold of the form $\RR_+\times M$, where $M$ is a  3-manifold. The spacetime structure on $\RR_+\times M$ is proposed to be a Freidman-Lema\^itre-Robertson-Walker metric 
\begin{align}\label{flrw-metric}
  -dt^2 +A(t)g(x),  
\end{align} 
where $A:\RR_+\times M\to \RR$ is a positive function and $g$ is a Riemannian metric on $M$. It is not yet known what are the topological and curvature properties of the spatial manifold $(M,g)$ which best fits gathered cosmological data. However, new cosmic microwave background radiation (CMBR) data obtained by the European Space Agency's Planck satellite has implicated that the manifold $(M,g)$ has nonzero curvature \cite{PhysClosed, PhysOpen}. Further, using the most recent Planck measurements, it is argued in \cite{PhysClosed} and \cite{luminet2016survey} that a closed manifold $(M,g)$ best fits the data. Earlier literature using CMBR data from NASA's WMAP satellite have also suggested that $(M,g)$ might be a closed 3-manifold  \cite{luminet2008shape}, \cite{luminet-finite}, \cite{lachieze1995cosmic}. In particular, Luminet et. al. have contended that the correlations in the cosmic microwave background data available in 2003 were consistent with the universe having the topology of the Poincar\'e dodecahedral space \cite{luminet-finite}.

In this paper, we propose a method to construct a laboratory device which represents the approximate optical properties of a closed, orientable, $C^{\infty}$-smooth Riemannian 3-manifold $(M,g)$. For such manifolds, we  consider a spacetime model of the universe given by a Lorentzian manifold of the form $(\RR_+\times M,- dt^2+g)$ (here we have set $A(t) =1$ in \eqref{flrw-metric}). The existence of such a device would enable cosmologists to both simulate the optical properties of a candidate manifold $(M,g)$ for the universe structure and compare this information to experimental data.
In addition to devices which simulate cosmological models, one could possibly modify the proposed devices to obtain
new optical instruments or resonators that modify waves in interesting ways, similar to  
field rotators \cite{ChenRotate};  illusion optics \cite{ChenChan2007,Lai2009};
  invisible sensors \cite {AE2,GKLU-physic}; perfect absorbers \cite{Landy2008};  cloaked wave amplifiers \cite{GKLU7}; and superdimensional resonators
  which spectrum imitates higher dimensional objects \cite{GKLU-superdim}.


The device we propose is a cloaking device based on the principle of transformation optics. An object is \emph{cloaked} if it is rendered invisible to measurements. Devices which electromagnetically cloak an object have been built using metamaterials \cite{AluEngheta2005}, \cite{Pendry2006}, \cite{ChenChan2007}. These so-called metamaterials are composite materials designed to manipulate the properties of electromagnetic radiation, so that the phase velocity of the waves with an (angular) wave number $k$ may be a very large or small. One way to prescribe the properties and arrangement of the metamaterials for a cloaking device is through the method of \emph{transformation optics}. Here, the viewpoint is that the metamaterials which comprise the cloaking device mimic the properties of an abstract Riemannian manifold, which is called \emph{virtual space}. To realize these waves in the laboratory, a piecewise smooth diffeomorphism, called the transformation map, from virtual space into $\RR^3$ is constructed. The image of the transformation map is called \emph{physical space}, since it is a representation of the virtual space manifold in the ``real'' 3-dimensional world of $\RR^{3}$. Via the transformation map, one could in practice construct the cloaking device in a laboratory, as the Riemannian metric induced on physical space from the pushforward of the Riemannian metric on virtual space by the transformation map dictates the necessary metamaterial properties, whereas the transformation map itself describes their arrangement in $\RR^{3}$.  

Transformation maps which blow up a point in an Euclidean ball were first proposed in \cite{GLU-nonunique}, in the context of demonstrating nonuniqueness to Calder\'on's problem; please see reviews on this and related inverse problems in \cite{Uhlmann-ICM,Uhlmann-Bull}. Extensions and perturbations of this type of transformation map construction have been widely promoted for the manufacturing of cloaking devices. 
Moreover, the mathematical feasibility of such cloaks has been studied in various physical and theoretical settings. For optical waves governed by a Helmholtz equation, a non-exhaustive list of the literature wherein the transformation-based invisibility cloaks are defined in various physical settings is
\cite{Astala,GKLU-fullwave,GKLU-quantum,KOVW2010,KOVW-IP,Weder1,Weder2}. The stability and accuracy of such invisibility cloaks has been analyzed in  \cite{DLU2017,LT2011-2,LiuUhlmann2015,LLLW2017,LiuZhou,LZ2011,LiuSun2013,LLRU2015,Nguyen2013,NV-siam,NV-arch,Nguyen2012}.  
A survey of some of these results and others is contained in the papers \cite{GKLU-invisibility,GKLU-survey,Uhlmann-review}. For electromagnetic waves described by Maxwell's equations, the mathematical analysis of cloaks in this setting has been addressed in  \cite{DLU2017-2,GKLU-fullwave,GKLU-Phys,GKLU-wormhole,LZ2016} and the stability of such cloaks has been studied in \cite{BaoLiu2014, BaoLiuZou}. Besides cloaks described by transformation optics, there are other alternatives to implement invisibility cloaks. In particular, we mention the mathematical analysis of plasmonic cloaking in the works of \cite{AluEngheta2005,Ammari2016,
Ammari-etal2013,AmmariPart1,AmmariPart2,MiltonMcPhedran,MiltonNicorovici}, and the discrete cloaking models contained in \cite{LST2015}.

We highlight that the transformation cloaking device given in \cite{GKLU-Phys} and \cite{GKLU-wormhole} for electromagnetic waves obeying Maxwell's equations models the electromagnetic properties of a virtual space with the geometry of a wormhole. The authors of these papers achieved a invisibility cloak with this geometry by building a transformation map which ``blows up a curve''. This construction led to the experimental implementation of a magnetic wormhole by physicists in 2015 \cite{Sanchez}. In this experimental implementation, the required metamaterial was built using superconducting elements.

Inspired \cite{GKLU-Phys} and \cite{GKLU-wormhole} we define a transformation map which allows us to render a physical space model $\tM\subset \RR^3$, endowed with a metric $\tilde g$, for any virtual space given by a closed, orientable, $C^\infty$-smooth, 3-manifold $M$.

To do this, we demonstrate and exploit the existence of a link in $M$ that is compatible with the differentiable structure of the manifold. We call a union $L$ of finitely many, pair-wise disjoint, embedded, smooth, simple closed curves $S_j$, $j=1,2,\dots, J$ in $M$ a \emph{link in $M$},
that is $L=\bigcup_{j=1}^J S_j$.  Assume for now that $M$ contains a link $L$, and that there is a smooth diffeomorphism
\begin{eqnarray}\label{links removed}
\Psi: M\setminus L \to \tM \subset \RR^3.
\end{eqnarray}
Here the set $\tM$ has the form
\begin{eqnarray}\label{model device}
\tM = \tM_0^{int} \setminus \bigcup_{j=1}^J  \widetilde  T_j,
\end{eqnarray}
where $\tM_0 \subset \RR^3$ is an open set having $2$-torus as its boundary (recall that a $2$-torus is a surface homeomorphic to $\SSS^1\times \SSS^1$ and a solid $3$-torus is homeomorphic to $\overline{B}^2\times \SSS^1$). Additionally in \eqref{model device}, $\tM^{int}_0$ is the set of interior points of $\tM_0$ and the manifolds $ \widetilde  T_j\subset \RR^3$, $j=1,\ldots J$, are mutually disjoint solid $3$-tori contained in $\tM_0$. 
The number tori appearing in the decomposition, $J$, is equal to the number of curves in the link minus 1.
Please see Figures \ref{fig:tori-cross-section} and \ref{fig:see-through}.


The existence of a diffeomorphism $\Psi$ as in  \eqref{links removed} is closely related to the question of whether the closed 3-manifold $M$ can be represented as a union of solid 3-tori and a set $N$ which is diffeomorphic to an open set in $\RR^3$. 
To consider this, let $T_j(R)$ be the closure of the 
$R$-neighborhoods of the curves $S_j$ on $M$, where $R>0$ is so small 
that each $  T_j(R)$ is diffeomorphic to a solid 3-torus.
We can write $M$ as   
\begin{eqnarray}\label{model on manifold}
M= N\cup \bigcup_{j=1}^J   T_j(R),
\quad \hbox{where }N=M \setminus \bigcup_{j=1}^J   T_j(R),
\end{eqnarray}
Clearly, when  a diffeomorphism $\Psi$ in  \eqref{links removed} exists, $N$ is diffeomorphic
to the open set $\widetilde N=\Psi(N)\subset \RR^3$, and $ T_j(R)$ are solid 3-tori.
Below, the proof of Proposition \ref{smooth-embedding} on
the existence of a diffeomorphism $\Psi$ in  \eqref{links removed} is closely related to the representation of $N$ as a union of a set $N$ diffeomorphic to  $\widetilde N\subset \RR^3$ and the solid 3-tori.  

%

As concrete examples of manifolds with representation \eqref{model on manifold}, consider the classical examples $\SSS^1\times \SSS^1$ and $\SSS^3$. Both spaces are obtained by gluing together two copies of the solid $3$-torus $\overline{B^2}\times \SSS^1$ along the boundary $\partial \overline{B^2} \times \SSS^1$. In both cases we have $J=1$ and the link $L$ is merely a single closed curve which is not knotted. 
In the case of $\SSS^1\times \SSS^1$, let $D_+$ and $D_-$ be the (closed) upper and lower hemispheres of $\SSS^1$, respectively. Then $\SSS^1 \times \SSS^1 = (D_+ \cup D_-)\times \SSS^1 = (D_+ \times \SSS^1) \cup (D_- \times \SSS^1)$. Hence,
$\SSS^1 \times \SSS^1$ is a union of two 3-tori.
To see that the $3$-sphere $\SSS^3$ is also a union of two copies of $\overline{B^2}\times \SSS^1$, we observe that, since $\SSS^3$ as the manifold boundary of the closed $4$-ball $\overline{B^4}$ and $\overline{B^4} \cong \overline{B^2}\times \overline{B^2}$, where $\cong$ denotes a homeomorphism between the spaces, we see that $\SSS^3 = \partial \overline{B^4}\cong (\partial \overline{B^2}\times \overline{B^2}) \cup  (\overline{B^2}\cup \partial \overline{B^2}) \cong (\SSS^1 \times \overline{B^2}) \cup (\overline{B^2}\times \SSS^1)$. 
Hence, also $\SSS^3$ is a union of two 3-tori.



In these cases, $\tM_0$ is a solid $3$-torus and $J=0$, since it is not necessary to remove additional tori. As a third and slightly more complicated example, the $3$-torus $\SSS^1\times \SSS^1\times \SSS^1$ may be obtained by choosing $L$ on $\SSS^3$ to be the Borromean rings, that are three suitably chosen closed curves \cite[Example 8.7]{Prasolov-Sossinsky-book}. Let $ \widetilde  T_0, \widetilde  T_1, \widetilde  T_2$ be the disjoint, solid $3$-tori obtained by constructing a tubular neighbourhood for each of the three circles comprising the Borromean rings on $\SSS^3$. Relating this description to our notation in \eqref{model device}, $\tM_0$ is the complement of an open 3-torus in $\SSS^3$. An example where $\tM_0$ is not a solid torus (but a domain which boundary is a 2-torus) is the Poincar\'e homology sphere. For this manifold, we may take the link $L$ to be the right trefoil knot. We refer to the books of Prasolov and Sosssinsky \cite{Prasolov-Sossinsky-book} and Saveliev \cite{Saveliev-book} for a detailed discussion on the surgery description of $3$-manifolds, as well as diagrams and more details concerning the examples we presented. A visualization of an internally knotted 3-torus which boundary is a 2-torus (as well as other 3-manifolds) has been artistically rendered by Carlo S\`equin in \cite{Sequin}\footnote{See a 3D printed model of an internally knotted torus at the page  {\url{http://gallery.bridgesmathart.org/exhibitions/2011-bridges-conference/sequin}}.}



\begin{figure}[ht]
    \centering
    \includegraphics[scale = 0.5]{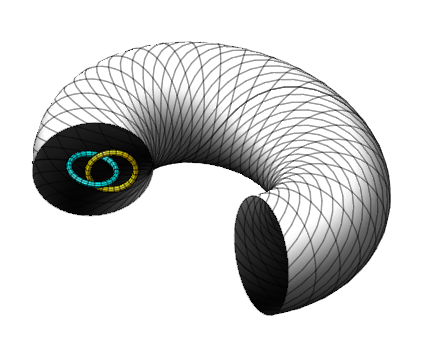}
    \caption{A depiction of a cross section the manifold $\tM$ when $\widetilde M_0$ is a solid $3$-torus $ \widetilde  T_0$. The larger torus is $\tM_0$ in \eqref{model device}, whereas the smaller tori represent $\partial  \widetilde  T_1$ and $\partial  \widetilde  T_2$ in \eqref{model device}.}
    \label{fig:tori-cross-section}
\end{figure}

\begin{figure}[ht]
    \centering
    \includegraphics[scale = 0.5]{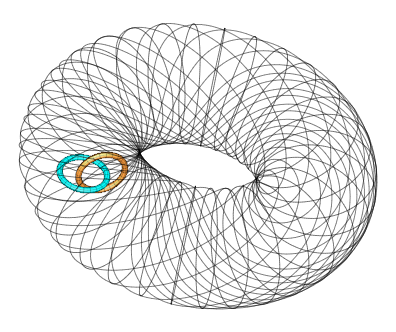}
     \includegraphics[scale = 0.5]{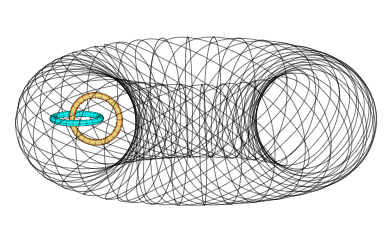}
    \caption{More visualizations of a manifold $\tM$ as described by \eqref{model device} when $\widetilde M_0$ is a solid $3$-torus $ \widetilde  T_0$. The larger torus $\tM_0$ has been made see-through; the smaller tori represent $\partial  \widetilde  T_1$ and $\partial  \widetilde  T_2$ in \eqref{model device}. The ``twisted'' coordinate lines on the boundary of the tori seen here arise from the map $\Psi$. These lines are given by Fermi coordinates defined near the link $L$ which are pulled back to $\widetilde M$ by $\Psi^{-1}$ (see Equation~\eqref{explicit diffeo} and Figure~\ref{fig:gluing2d}). This ``twisting'' of the boundaries can be viewed as a demonstration that the metamaterial corresponding to the metric $\tg$ is anisotropic in directions which follow these coordinate lines.}
    \label{fig:see-through}
\end{figure}

Returning to our transformation map construction of a cloaking device, we equip the manifold $\tM$ with the pullback metric $\tg:= (\Psi^{-1})^*g$. Then, $(\tM,\tg)$ is the physical space candidate for a cloaking device which describes static optical waves in the universe model $(\RR_+\times M, -dt^2+g)$. In particular, the map $\Psi$ is used to define the properties of a cloaking device prescribed by $(\tM,\tg)$ which recreates a good approximation of the optics of time-harmonic waves on the spatial manifold $(M,g)$. We note that
$(\RR_+\times (M\setminus L), -dt^2+g)$ is isometric Lorentzian manifold to $(\RR_+\times \tM, -dt^2+\tg)$.
Time-harmonic waves on $(M,g)$ are then solutions $U(t,x) = e^{-ikt}u(x)$, to the wave equation
\[
\partial^2_tU(t,x) - \Delta_g U(t,x) =e^{-ikt}f(x)
\]
on $\RR_+\times M$, where $\Delta_g$ is the Laplace-Beltrami operator associated to $g$, $k\in \mathbb{C}$ is the wavenumber, $f:M\to \CC$ is a source of the field, and $u:M\to \CC$ solves the Helmholtz equation 
\begin{align}\label{helmholtz}
\Delta_gu+k^2u= \frac{1}{\sqrt{|\det g|}}\sum_{a,b=1}^{3}\partial_a\left(\sqrt{|\det g|}g^{ab}\partial_bu\right)+ k^2u=f
\end{align}
on $M$. Here $g^{ab}$ are the components of the inverse metric $g^{-1}: T^{*}M\times T^{*}M\to [0,\infty)$. The symmetric tensor $\sigma$ with components
\begin{align}\label{conductivity}
\sigma^{ab}(s) = \sqrt{|\det g(s)|}g^{ab}(s)
\end{align}
is called the \emph{conductivity} (in the electrostatic case). We note that the conductivity tensor is in fact a product of a 2-contravariant tensor and a half density. If $V\subset M$ is an open set such that $\overline V\subset M\setminus L$ and $f$ is supported in $\overline V$, we define the \emph{source-to-solution map} as
\begin{align}\label{ss-map-M}
\Lambda_{V}:L^{2}(V,g) \to L^{2}(V,g), \ \ \Lambda_{V}(f) = u|_{V},
\end{align}
where $u$ solves \eqref{helmholtz}. Henceforth in this paper we will consider $L^2(V,g)$ as the subspace of $L^2(M,g)$ consisting of functions with support over $\overline{V}$. The source-to-solution map models the local measurements, where we implement a source $f$ in the set $V$ and observe the produced wave $u$ in the set $V$.

To show our transformation map gives an appropriate prescription of $\tg =\Psi_*g$
we establish that time-harmonic waves in the physical space $(\tM,\tg)$ stably reproduce waves in the virtual space $(M,g)$.
This will be achieved despite the fact that the metric $\tg$ will turn out to be degenerate near the boundary $ \Sigma =\partial \tM$.

To analyze the reproduction stability, for $\eps>0$ let $T(\eps)$ be a closed tubular neighbourhood (with respect to the metric $g$) of radius $\eps$ of the link $L\subset M$. Then, we may obtain an $\eps$-neighbourhood (with respect to the metric $\tg$) of $\Sigma = \partial \tM$ given by $\tT(\eps)=\Psi(T(\eps)\setminus L)$. Write $\tilde\nu_\epsilon$ for the unit outward pointing normal vector field on $ \partial \tT(\eps)$. We consider the Helmholtz equation with Neumann boundary condition,
\begin{equation}\label{e-solutions}
\begin{aligned}
\Delta_{\tg} \tu_\epsilon + k^2\tu_\epsilon &= \tf &&\text{on }\tM\setminus \tT(\eps)\\
\partial_{\tilde{\nu}_\eps}\tu_\epsilon &= 0 &&\text{on } \partial \tT(\eps).
\end{aligned}
\end{equation}
In this setting, if $\widetilde{V}$ is an open, relatively compact subset of $\tM\setminus \overline{\tT(\eps)}$ and $\tf$ is supported in the closure of $\widetilde{V}$, the \emph{source-to-solution map} associated to the problem \eqref{e-solutions} is defined as
\begin{align}
\widetilde{\Lambda}_{\tV,\eps}: L^{2}(\tV,\tg) \to L^{2}(\tV,\tg), \ \ \widetilde{\Lambda}_{\tV,\eps}(\tf) = \tu_{\eps}|_{\tV},\label{ss-map-eps-tilde}
\end{align}
where $\tu_{\eps}$ solves \eqref{e-solutions} with source $\tf$.  

The following theorem is our main result, which roughly speaking states that in the device we design by transformation optics it is possible to simulate, with arbitrary precision, the time-harmonic waves in a static universe that are produced by any compactly supported source. 
\begin{theorem}\label{main-theorem}
Let $(M,g)$ be a closed, orientable, $C^{\infty}$-smooth Riemannian 3-manifold,
 $L\subset M$ be a link (i.e.\ a union of finitely many, pair-wise disjoint, embedded, smooth, simple closed curves) and $\Psi: M\setminus L \to \tM$ be a smooth diffeomorphism where $\tM \subset \RR^3$ is a bounded open set with a smooth boundary.
Set $\widetilde{V}\subset \tM$ to be a relatively compact open set, $V=\Psi^{-1}(\widetilde V)$, and let $\tf\in L^{2}(\widetilde{V},\tg)$. Suppose that $-k^{2}$ is not an eigenvalue of the Laplace operator $\Delta_{g}$ on $(M,g)$. Then, 
\[ 
\lim_{\eps\to 0}\widetilde{\Lambda}_{\tV,\eps}\tf =(\Psi^{-1})^*( \Lambda_{V}(\Psi^{*}\tf))
\]
in $L^{2}(\widetilde{V},\tg)$.

Moreover, if 
$0<\alpha<1$ and $m\in \mathbb N$ satisfy $m+\alpha >1/2$ and 
$f\in C^{m,\alpha}_0(\tV)$, then 
\[ 
\lim_{\eps\to 0}\widetilde{\Lambda}_{\tV,\eps}\tf =(\Psi^{-1})^* (\Lambda_{V}(\Psi^{*}\tf))
\]
in $C^{m+2, \alpha}(\text{cl}_{\tM}\widetilde{V})$, where $\text{cl}_{\tM}\widetilde{V}$ is the closure of $\widetilde V$ in $\tM$.
\end{theorem}

Earlier, we postponed the discussion on the existence of a link in $M$ in \eqref{model device} and the corresponding existence of the transformation map $\Psi$ used Theorem \eqref{main-theorem}. We return to these points now. In this paper we show that we may obtain both a link $L$ and a link blow-up map $\Psi$ as a consequence of the following extension of the classical Lickorish--Wallace theorem. The theorem we prove states that the class of the homeomorphism performing the Lickorish--Wallace surgery contains a $C^\infty$-smooth map whose differential satisfies a technical, explicit bound; the precise statement we placed in Section \ref{const-3manifold}, Theorem \ref{thm:smooth-embedding-precise}. A consequence of our result is the following proposition:
\medskip

\begin{prop}\label{smooth-embedding}
Let $M$ be a smooth, closed, oriented, and connected Riemannian $3$-manifold. Then, there exists a smooth link $L\subset M$ and a $C^\infty(M\setminus L)$-smooth embedding $\Psi: M\setminus L \to \RR^3$. 
\end{prop}

\subsection{Connection to invisibility cloaking and black holes}

Recall that an object is cloaked if it is rendered invisible to measurements. In Theorem \ref{main-theorem} we used a cloaking device to electromagnetically model a static spacetime. We now present an example which illustrates the invisibility aspect of our cloaking construction.

In Theorem \ref{main-theorem}, the physical space $\tM\subset \mathbb R^3$ has the form of \eqref{model device}; hence it is a bounded domain with a non-empty (and possibly disconnected) boundary $\Sigma$. Consider an idealized situation where the shape $\tM$ is filled with a metamaterial such that the phase speed of electromagnetic waves through $\tM$ is described by  a Riemannian metric $\tg$ which is degenerate near the boundary $\Sigma$. From this idealized setting, we can build a more realistic model with a non-degenerate metric as follows. Take $\eps>0$ and consider the domain 
\[
\tM \setminus \tT(\eps)\subset  \tM \subset \mathbb R^3,
\]
where $\tT(\eps)$ is an $\eps$-neighbourhood about $\Sigma$. Again, assume that the domain $\tM \setminus \tT(\eps)$ has been filled with a metamaterial constructed so that the phase speed of the waves is given by the Riemannian metric $\tg|_{\tM \setminus \tT(\eps)}$.  As $\tg|_{\tM \setminus \tT(\eps)}$ is non-degenerate, phase speeds of waves in $\tM\setminus \tT(\eps)$ are bounded from below and above by positive constants. Next, cover the boundary $\Sigma(\eps) =\partial (\tM\setminus \tT(\eps))$ with a reflective material that corresponds to the Neumann-boundary condition  
\[
\partial_{\nu_{\eps}}u = 0\text{ on }\Sigma(\eps).
\]

Now take the observational set $\tV$ of Theorem \ref{main-theorem} to be a large open set in $\tM$, for instance large enough so that $(\tM\setminus \tT(2\eps))\subset \tV$. 
Thus, via Theorem \ref{main-theorem}, if we measure in $\tV$ the wave  $\tilde u$ produced by a source $\tilde f$, the result of this measurement is arbitrarily close (depending on $\eps$) to the results which one would obtain on doing measurements on the set $V = \Psi^{-1}(\tV)$ in the closed manifold $(M,g)$.  In particular, the metamaterial in $\tM\setminus \tT(\eps)$ makes the boundary of the domain and the ``external'' space,  $\mathbb R^3\setminus (\tM\setminus \tT(\eps))$ invisible in these measurements.

To facilitate a visualization of the above example, we make a playful pop culture connection to Theorem \ref{main-theorem}: Assume that the fictional character Harry Potter makes a torus-shaped bag from his cloak of invisibility (which we suppose is build up from metamaterials). Harry Potter then hides inside the bag, thus making the external world invisible.
From the perspective of Theorem \ref{main-theorem}, his observations inside the torus-shaped bag would be similar to those he would make on a closed manifold (in fact, in the manifold $\mathbb \SSS^3$).

We were motivated to propose the cloaking device prescribed by the transformation map $\Psi$ by recent Plank satellite data, the validation of which is given my Theorem \ref{main-theorem}. Our cloaking device aims to simulate the optical behaviour of geometries prescribed by \eqref{model device}. However, it may also enable the simulation of entangled black holes. One may view surgery over a curve (or more generally a link) in a 3-manifold as constructing a manifold with a wormhole geometry \cite{AKL-cosmic}, \cite{AKL-BH}; please also see \cite{GKLU-wormhole} for a concrete example. It is suggested in \cite{AKL-cosmic}, \cite{AKL-BH} that there may be a connection between entangled black holes and wormhole manifolds such as those we consider in \eqref{model device}. Thus, our cloaking device construction may be used to provide an optical device which simulates the entangled geometry of the black holes.

\subsection{Overview}
In Section \ref{universe-model}, we motivate our proof of Proposition \ref{smooth-embedding} by considering a similar situation in two dimensions. We prove that the analogous result to Proposition \ref{smooth-embedding} does not hold in two dimensions. The proof of Proposition \ref{smooth-embedding} is deferred to Section \ref{const-3manifold}. In Section \ref{universe-model} we also establish estimates for the metric $\tg$ on our virtual space. 

Section \ref{thm-proof-section} contains the set up and preliminary work to prove Theorem \ref{main-theorem}. First, in Section \ref{function-spaces} we define the function spaces we work with and establish equivalences between Sobolev spaces on our virtual space $(M,g)$ and those of the same order on $(M\setminus L,g)$ and our physical space $(\tM,\tg)$. Then, in Section \ref{fried-ext}, we define Friedrichs extensions of the operators appearing in \eqref{helmholtz} and \eqref{e-solutions}, as well as their corresponding sesquilinear forms. Section \ref{gamma-section} establishes $\Gamma$-convergence results for the forms in Section \ref{fried-ext}. These convergence results allow us to prove Theorem \ref{main-theorem} in the case where the wavenumber appearing in \eqref{helmholtz} and \eqref{e-solutions} is negative. As we are concerned with wave solutions, in Section \ref{resolvent-section}, we establish convergence of the resolvents associated to the Friedrich's extensions of $\Delta_g$. Lastly, in Section \ref{the-main-proof} we use the convergence results in Section \ref{resolvent-section} to prove Theorem \ref{main-theorem}.

\textbf{Acknowledgements.} TB and ML were partially supported by Academy of Finland, grants 273979, 284715, 312110, and Finnish Centre of Excellence in Inverse Modelling and Imaging.
The PP and VS have been partially supported by Academy of Finland, grants 297256 and 322671.

%
%

\section{Virtual and physical spatial models}\label{universe-model}


In this section, we describe our construction of a piecewise smooth embedding $M\setminus L\to \RR^3$, where $L$ is a link in $M$. This embedding will be a  composition of a piecewise smooth embedding $M \setminus L \to \SSS^3$ and a stereographic projection map. Here we present a heuristic overview of the construction of how we obtain the embedding $M\setminus L\to \SSS^3$, and place the details of the argument in Section \ref{const-3manifold}. To begin, we recall the concept of a link in $M$:
\begin{defn}\label{link-defn}
A subset $L$ is a \emph{link in a manifold $M$} if $L$ is a union of finitely many, pair-wise disjoint embedded circles $S_1,\ldots, S_J$. That is, $L = \bigcup_{j=1}^J S_j$ where $S_j \cap S_k = \emptyset$ for $j\ne k$, and for each $j=1,\ldots, J$, there exists an embedding $\phi_j \colon \SSS^1 \to M$ for which $S_j = \phi_j(\SSS^1)$. Further, we say that a link $L$ in $M$ is \emph{smooth} if we may take each embedding $\phi_j$ to be a smooth embedding. 
\end{defn}
If there exists a link $L\subset M$, we may define a compact tubular neighbourhood $T$ of this link which is diffeomorphic to a union of pairwise disjoint solid tori. Then, we perform a \emph{surgery} of $M$ along the boundary $\partial T$. That is, we remove the interior of $T$ and transform $T$ to a homeomorphic copy $T'$. Then we glue $T'$ to $M\setminus \text{int}(T)$ along the boundary of $T'$. This is analogous to the situation in $2$-dimensions, where a pair of disks on a surface is replaced with a cylinder to obtain a surface with higher genus. Below, we explain how to use surgery to construct a piecewise smooth homeomorphism in the 2-dimensional case to motivate the role surgery near the link $L\subset M$ plays in the construction of the embedding in the 3-dimensional case.



%
%

\subsection{An excursion: $2$-dimensional embedding theorem}

We define a \emph{surface} as a $2$-manifold with or without boundary. The boundary of a surface will always refer to its manifold boundary. 

Closed and orientable surfaces are classified by their genus; we do not define the notion of genus here, please see~\cite{Hatcher} for the appropriate details. We use the genus to represent a closed and orientable surface as the boundary of a $3$-dimensional manifold in $\RR^3$ as follows. 

Recall that the $2$-sphere $\SSS^2\subset \RR^3$, which is the boundary of the closed unit $3$-ball $\bar B^3$, is a surface of genus $0$. Let now $J\in \ZZ_+$. For $j=1,\ldots, J$, we fix embeddings $\psi_j \colon \bar B^2 \times [0,1] \to \RR^3$ having mutually disjoint images which satisfy
\[
\phi_j(\bar B^2 \times [0,1]) \cap \bar B^3 = \phi_j(\bar B^2\times \{0,1\}) \subset \SSS^2
\]
for each $j=1,\ldots, J$. Then 
\[
M = \bar B^3 \cup \phi_1(\bar B^2 \times \{0,1\}) \cup \cdots \cup \phi_J(\bar B^2 \times \{0,1\})
\]
is a $3$-manifold with boundary in $\RR^3$ and the boundary $\Sigma = \partial M$ is a surface of genus $J$. 
Typically, the images of embeddings $\phi_i$ are called \emph{handles of $M$} and the manifold $M$ a \emph{handlebody}.

Having the surface $\Sigma$ at our disposal, we may formulate a $2$-dimensional analog of Proposition \ref{smooth-embedding} as follows.

\begin{prop}
\label{prop:2-dim-embedding}
Let $\Sigma$ be a surface of genus $J\in \ZZ_+$. Then there exists a link $L \subset \Sigma$ consisting of $J$ circles, pair-wise disjoint closed disks $D_1,\ldots, D_{2J} \subset \SSS^2$, and a homeomorphism $\Sigma \setminus L \to \SSS^2 \setminus (D_1\cup \cdots \cup D_{2J})$. 
\end{prop} 

A careful reader will recall that the embedding in Proposition \ref{smooth-embedding} is smooth, whereas the embedding $\Sigma\setminus L \to \SSS^2$ in Proposition \ref{prop:2-dim-embedding}, induced by the homeomorphism, is merely topological. The existence of a smooth embedding follows from the classical results of Rad\'o on Riemann surfaces, which states that each topological surface carries a unique smooth structure which in turn implies that homeomorphic smooth surfaces are diffeomorphic.

\begin{proof}[Proof of Proposition \ref{prop:2-dim-embedding}]
We may assume that $\Sigma=\partial M$, where $M$ is the genus $J$ handlebody constructed above. Let also embeddings $\phi_i \colon \bar B^2 \times [0,1] \to \RR^3$ be as in the construction of $M$ for $j=1,\ldots, J$.

For each $j=1,\ldots, J$, let
\[
E_{2j+k} = \phi_j(B^2 \times \{k\}),\;
D_{2j+k} = \phi_j(\bar B^2(1/2) \times \{k\}),\;
\text{and} \;
H_j = \phi_j(\SSS^1\times (0,1)).
\]
First observe that
\[
\Sigma \cap \SSS^2 = \SSS^2 \setminus \left(E_1\cup \cdots \cup E_{2J}\right)
\]
and 
\[
\Sigma = (\Sigma \cap \SSS^2) \cup \left( H_1 \cup \cdots \cup H_J\right);
\]
thus heuristically $\Sigma$ is obtained by removing open disks $E_1,\ldots, E_{2J}$ from $\SSS^2$ and attaching open cylinders $H_1,\ldots, H_J$. 

Let $L \subset \Sigma$ be the link with circles
\[
S_j = \phi(\SSS^1 \times \{1/2\}) \subset H_j
\]
for $j=1,\ldots, J$. Let also 
\[
H_{j-} = \phi_j(B^2\times (0,1/2))\quad \text{and} \quad H_{j+} = \phi_j(B^2\times (1/2,1))
\]
be open cylinders for each $j=1,\ldots, J$; note that $H_j \setminus S_j = H_{j-}\cup H_{j+}$.

We define now a homeomorphism $h \colon \Sigma \setminus L \to \SSS^2 \setminus (D_1\cup \cdots \cup D_{2J})$ as follows. On $\Sigma \cap \SSS^2$, we set $h$ to be the identity. Then, for each $j=1,\ldots,J$, we take, $h|_{H_{j-}}$ to be the unique embedding satisfying 
\[
h(\phi_j(x,t)) = \phi_j((1-t)x,0)
\]
for $(x,t) \in \bar B^2 \times (0,1/2)$. Similarly, we take $h|_{H_{j+}}$ to be the unique embedding satisfying
\[
h(\phi_j(x,t)) = \phi_j(tx,1)
\]
for $(x,t)\in \bar B^2\times (1/2,1)$. 
This completes the proof.
\end{proof}

We remark that for two dimensional manifolds there are invisibility cloaking results where one removes the origin from a two dimensional disc $\mathbb D^2$ of radius $2$ and center $0$ and uses the transformation optics approach with a blow-up diffeomorphism $F:\mathbb D^2\setminus \{0\}\to
\mathbb D^2\setminus \bar{ \mathbb D}^1$, where $\mathbb D^1$ is a disc of radius $1$. These cloaking results are based on the fact that the single point $\{0\}$ has the zero capacitance in $\mathbb D^2$. In light of this, it would be natural to ask if one may construct similar two dimensional models where one removes a set of zero capacitance. In particular, let $S$ be a two dimensional closed manifold and $\{p_1,p_2,\dots,p_J\}\subset S$ be a finite number of points. Could one build a transformation map from $S\setminus\{p_1,p_2,\dots,p_J\}$ to the two-dimensional plane $\RR^2$ such that optical waves on $S\setminus\{p_1,p_2,\dots,p_J\}$ are equivalent to those on the image of $S\setminus\{p_1,p_2,\dots,p_J\}$ in $\RR^2$?


As the capacitance zero sets on $S$ (such as $\{p_1,p_2,\dots,p_J\}$) have Hausdorff dimension zero, the following theorem shows that unlike in the 3-dimensional case of Theorem \ref{smooth-embedding}, in the 2-dimensional case we cannot achieve an embedding into $\RR^2$ by removing a subset of capacitance zero (or generally, a set which Hausdorff dimension is less than one) from a manifold that is not homeomorphic to a sphere. 

\begin{theorem}
Let $S$ be a Riemann surface of positive genus $g>0$ and $E\subset S$ a set of Hausdorff $1$-measure zero, that is, $\haus^1(E)=0$. Then there is no embedding $S\setminus E\to \RR^2$.
\end{theorem}

\begin{remark}
Note that, although the Riemannian metric is not fixed on $S$, the null sets of Hausdorff $1$-measure are well-defined on $S$.
\end{remark}

\begin{proof}
By the Uniformization theorem, there exists a covering map $\pi \colon X \to S$, from the universal cover $X$ to $S$, which is a conformal local diffeomorphism and where $X$ is the Euclidean space $\CC$ in the case $g=1$ or the hyperbolic disk $\mathbb D$ for $g\ge 2$. Let also $\Gamma$ be the deck group of isometries of $\pi$ and let $\Omega$ be a fundamental domain for $\Gamma$. Then $\Omega \subset X$ is a $2g$-gon, whose opposite faces the covering map $\pi$ identifies.

Since $\pi$ is conformal local diffeomorphism, we have that the set $\widetilde E = \pi^{-1}(E)\cap \Omega$ has Hausdorff $1$-measure zero. 

Given that $g>0$, we may fix two different pairs $\alpha,\alpha'$ and $\beta,\beta'$ of opposite faces of $\Omega$ for which the interiors of $\alpha\cup \alpha'$ and $\beta\cup \beta'$ on $\partial \Omega$ do not meet.
 
We fix now line segments $\ell_\alpha \subset \Omega$ and $\ell_\beta \subset \Omega$ as follows. First, consider all line segments $\ell=[x,x']$ in $\Omega$ connecting faces $\alpha$ and $\alpha'$ for which points $x\in \alpha$ and $x'\in \alpha'$ are not corner points and $\pi(x)=\pi(x')$. 

From the fact that $\haus^1(\widetilde E)=0$, we may apply Fubini's theorem to obtain that there exists among these line segments a line segment $\ell_\alpha$ for which $\ell_\alpha \cap \widetilde E = \emptyset$. We fix a line segment $\ell_\beta$ analogously. To this end, we also record the observation that, since $\alpha\ne \beta$ and the end points of the line segments are not corner points, we have that $\ell_\alpha$ and $\ell_\beta$ meet exactly at one point.  

Let now $\gamma_\alpha = \pi(\ell_\alpha)$ and $\gamma_\beta = \pi(\ell_\beta)$. then $\gamma_\alpha$ and $\gamma_\beta$ are closed curves on $S$. Furthermore, since $\Omega$ is a fundamental domain, we have that $\gamma_\alpha$ and $\gamma_\beta$ are simple closed curves, that is, they are homeomorphic to $\mathbb S^1$. Moreover, $\gamma_\alpha \cap E = \gamma_\beta \cap E = \emptyset$ by construction.

Suppose now that there exists an embedding $f \colon S\setminus E \to \RR^2$. Then $f(\gamma_\alpha)$ is a Jordan curve in $\RR^2$. By the Jordan curve theorem, $\RR^2\setminus f(\gamma_\alpha)$ consists of two components $D$ and $D'$ and $\partial D = \partial D' = f(\gamma_\alpha)$. Since the line segment $\ell_\beta$ meets $\gamma_\alpha$ in $\Omega$ transversally in one point, the same holds for $\gamma_\beta$ and $\gamma_\alpha$ in $S$, and further for $f(\gamma_\beta)$ and $f(\gamma_\alpha)$ in $\RR^2$. Let $x_0\in \RR^2$ be the intersection point $f(\gamma_\beta)\cap f(\gamma_\alpha)$. Thus $f(\gamma_\beta)$ intersects non-trivially both sets $D$ and $D'$.

 Let $x_0\in \RR^2$ be the intersection point $f(\gamma_\beta)\cap f(\gamma_\alpha)$. Then $\{ D\cap f(\gamma_\beta), D'\cap f(\gamma_\beta)\}$ is a separation of $f(\gamma_\beta)\setminus \{x_0\}$. This is a contradiction $f(\gamma_\beta)$ is homeomorphic to $\mathbb S^1$ and $\mathbb S^1$ cannot be separated by one point. We conclude that such embedding $f$ does not exist.
\end{proof}


\subsection{Estimates for the physical model metric}\label{coord-est}

Given $M$, by Proposition \ref{smooth-embedding} or its refinement Theorem \ref{thm:smooth-embedding-precise} (please see Section \ref{const-3manifold} and the proof therein), there exists a link $L\subset M$, a domain $D\subset \SSS^3$, and a diffeomorphism $F: M\setminus L\to D$. Consider a point $p_0\in \SSS^3\setminus[D\cup \partial D]$, and let $\pi_{p_0}: \SSS^3\setminus\{p_0\}\to \RR^3$ be the map given by stereographic projection through the point $p_0$. The composition of these maps give a diffeomorphism $$\Psi:M\setminus L\to \tM\subset \RR^3,\ \ \Psi:= \pi_{p_0}\circ F,$$ where $\tM:= \pi_{p_0}(D)$ is a $C^\infty$-smooth manifold embedded in $\RR^{3}$. 

By setting 
\begin{align}
\tg &: T\tM\times T\tM\to \mathbb R_+,\quad \tg = (\Psi^{-1})^*g.
\end{align}
we achieve a Riemannian isometry 
\begin{align}\label{transformation-map}
\Psi:(M\setminus L,g)\to (\tM, \tilde g). 
\end{align}
The submanifold $\Sigma := \partial \tM\subset \RR^3$ is called the \emph{cloaking surface}. We emphasize that the metric $\tg$, represented in the Euclidean coordinates, is not assumed to be bounded from above or below near $\Sigma$. From this point forward, we will use a label with a tilde to denote an object associated to the physical space $(\tM,\tilde g)$ and untilded labels to denote objects associated to the virtual space $(M,g)$. To simplify the notation in this section and the following sections, when the context is clear, we will use $g$ to denote the Riemannian metric on either $M$ or the induced metric $g|_{M\setminus L}$ on $M\setminus L$. 

Recall that the optical and electromagnetic properties of the virtual space are captured by the Riemannian metric $g$, via the conductivity tensor $\sigma = \sqrt{|\det g|}g^{-1}$. In our physical space model $(\tM, \tilde g)$, the conductivity tensor is given by the pushforward of $\sigma$ by $\Psi$:
\begin{align}
\tsig(\tx) &:= \Psi_{*}\sigma(\tx) = \left.\frac{D\Psi(x) \sigma(x)D\Psi^{\top}(x)}{\left|\text{det}\,D\Psi(x)\right|}\right|_{x=\Psi^{-1}(\tx)} = \sqrt{|\det \tg(\tx)|}\,\tg^{-1}(\tx).
\end{align}
Above, $D\Psi$ denotes the differential of $\Psi$ and $D\Psi^{\top}$ denotes the transpose of the matrix $D\Psi$.

We shall see that the quantities $\tg(\tx)$, $\sqrt{|\det \tg(\tx)|}$, and $\tsig(\tx)$ become singular as $\tx\in \tM$ approaches $\Sigma$. In the next section we will demonstrate this behaviour by employing an advantageous coordinate system near $\Sigma$.



About each curve $ S_j\subset L$, we consider the (closed) tubular neighbourhood of radius $r>0$: 
\[
T_{j}(r):=\{x\in M\,:\,\text{dist}_{ g}(x,  S_j)\le  r\}.
\]
 Since the curves $ S_j$ are disjoint, there exists an $0<R\le1$ such that if $r <R$, the elements of the collection $\{T_{j}(r)\,:\, j=1,2,\dots,J\}$ are pairwise disjoint. We further restrict the value $R$ to be smaller than the injectivity radius of $L$ in $M$. To simplify notations, without loss of generality, we assume below that $R\ge 1$.

Then, for any $0<r \le R$, let
\begin{align}
T(r)&:= \bigcup_{j=1}^JT_{j}(r) \label{TubularL}
\end{align} 
be a tubular neighbourhood about the link $L$ and let $V_\theta$ be the inward-pointing unit normal vector to $\partial T(R)$ at $\phi(\{R\}\times\{s\}\times\{j\})$. Now, by our choice of $R$, for all $r\le R$ and for all $j=1,\ldots,J$, the map
 \begin{align*}
&\phi_{norm}: [0,R]\times\SSS^1\times \SSS^1\times\{1,\dots,J\}  \to T(R),\\
& \phi_{norm}(r,\theta,s,j) = \exp_{\phi(R,s,j)}( (R-r)V_\theta),
\end{align*}
is a diffeomorphism from $(0,R]\times\SSS^1\times \SSS^1\times\{1,\ldots,J\}$ to $T(R)\setminus L$. We call the coordinate system given by $x=(r,\theta,s,j)$ \emph{normal coordinates adapted to $L$}. Additionally, we denote by $\partial T(\eps)$ the boundary of $T(\eps)$ and by $\nu_\eps$ the outward-point unit normal vector field on $\partial T(\eps)$. 

Next, we turn our attention to the model $(\tM,\tg)$ and define an analogous coordinates system near the surface $\Sigma\subset \RR^3$ via the map $\Psi = \pi_{p_{0}}\circ F$. For this purpose, consider once again the tori $T(r)$, $0<r\le R$ defined in \eqref{TubularL}. From the map $\Psi$, we obtain corresponding  tori
 \begin{align}
\tT(r):=\{\tx\in\tM\,:\,\tx=\Psi(x), \, x\in T(r)\setminus L\} \subset \tM. \label{TubularSigma}
\end{align} 
We recall that $\Sigma$ is the boundary of the open set $\tM\subset \mathbb R^3$
so that $\Sigma$ is not a subset of $\tM$. However,
the union $\tT(r) = \cup_{j=1}^J\tT_j(r)\subset \tM$ can be considered as the set where
the tubular coordinates associated to $\Sigma$ are defined.

Since $F$ is a diffeomorphism away from the link $L$ and $\tg = (\Psi^{-1})^*g$, for $R>0$ as above the map 
\[ 
 \tilde{\phi}_{norm}:= \pi_{p_{0}}\circ \beta,\]
 where
\begin{align} \label{explicit diffeo}
&\beta:  \overline {B}^2\left(R\right)\setminus \overline {B}^2\left(1/2\right)\times \SSS^1\times\{1,\ldots,J\}\to \tT(R), \\
& \beta(\tr, \tth,s,j) = \exp_{F\circ\phi(R,s,j)}( (2\tr-1)F_{*}V_\theta),\nonumber\\
&\tr = \frac{1+R-r}{2}, \nonumber
\end{align}
is a diffeomorphism. We call the coordinates $(\tr, \tth,\ts,j)$ \emph{normal coordinates adapted to $\Sigma$}. Similar as before, we write $\partial \tT(\eps)$ for the boundary of $\tT(\eps)$ and by $\tilde{\nu}_\eps$ the outward-point unit normal vector field on $\partial \tT(\eps)$.

There is a constant $C>0$, depending on $\Psi$ and its derivatives, such that
in the above normal coordinates adapted to $\Sigma$, we have  
\begin{align*}
&|\tilde g^{a s} (\tx)| \le \frac{C}{(2\tr-1)} \text{ for } a\in\{s,\tth\},&&  |\tilde g^{a \tr} (\tx)|  \le C \text{ for } a\in\{\tr,s\},\\
&|\tilde g^{\tth\tth} (\tx)| \le \frac{C}{(2\tr-1)^{2}},
\end{align*}
where we denote $\tr=\tr(x)$.
Additionally, we have the estimate
\begin{align}
\sqrt{\tilde g}\le C(2\tr-1).\label{est-volume}
\end{align}

Combining the above inequalities, the conductivity $\tilde \sigma := \sqrt{\tilde g}\tilde{g}^{-1}$ satisfies
\begin{align}
|\tilde{\sigma}^{ab}(\tx)|&\le C(2\tr-1) \text{ for } a\in\{\tr,s\},\label{est-conductivity}\\
|\tilde{\sigma}^{a\theta}(\tx)|&\le C \text{ for } \text{ for } a\in\{s,\tth\},\label{est-conductivity2}\\
|\tilde \sigma^{\tth\tth} (\tx)| &\le \frac{C}{2\tr-1}, \label{est-conductivity-rad}
\end{align}
and hence becomes singular as $\tr\to \frac12$. 

\section{The behaviour of optical waves in virtual and physical space}\label{thm-proof-section}


In this section, we inspect the following Helmholtz problem on $(M\setminus L,g)$. Let $V$ be an open and relatively compact set in $M\setminus L$ such that the closure $\overline V$ satisfies $\overline{V}\cap L = \emptyset$. For $f\in L^{2}(M\setminus L,g)$ with support contained in $V$,
let $u\in H^{1}(M\setminus L,g)$ be a function which solves
\begin{align}\label{helmholtz2}
\Delta_gu+k^2u=f \ \ \text{on } M\setminus L
\end{align}
in sense of distributions.  Note that equation \eqref{helmholtz} does not involve any boundary conditions. In Section \ref{fried-ext} Lemma \ref{existence-helmholtz} we show that this problem is uniquely solvable. Therefore, we obtain a well-defined \emph{source-to-solution map} $\Lambda_{V,0}$ associated to problem \ref{helmholtz}:
\begin{align}\label{ss-map-ML}
\Lambda_{V,0}:L^{2}(V,g) \to L^{2}(V,g), \ \ \Lambda_{V,0}(f) = u|_{V},
\end{align}
where $u$ solves \eqref{helmholtz2}. 

We also study Helmholtz boundary value problems related to \eqref{helmholtz2}. Let $V$ and $f$ be as above. Additionally, let $\eps>0$, $T(\eps)\subset M$ denote the open tubular neighbourhoods of $L$ specified earlier by \eqref{TubularL}, and $\nu_{\eps}$ be the outward pointing unit normal vector field on $\partial T(\eps)$. Suppose that $u_\eps\in H^1(M\setminus T(\eps))$ solves
\begin{align}\label{helmholtz-eps}
\Delta_gu_{\eps}+k^2u_{\eps} &= f\quad \text{on } M\setminus T(\eps)\\
\partial_{\nu_{\eps}}u_{\eps} &= 0\quad \text{on } \partial T(\eps)\notag
\end{align}
in weak sense, see \cite{Evans-PDE}. In Section \ref{fried-ext}, we also show that this problem is uniquely solvable. Hence the associated source-to-solution map
\begin{align}\label{ss-map-eps}
\Lambda_{V,\eps}:L^{2}(V,g) \to L^{2}(V,g), \ \ \Lambda_{V,\eps}(f) = u_{\eps}|_{V},
\end{align}
where $u_\eps$ solves \ref{helmholtz-eps} is well-defined.


Once we show that the problems \eqref{helmholtz2} and \eqref{helmholtz-eps} are well defined, we use the transformation map $\Psi:M\setminus L\to \tM$ defined by \eqref{transformation-map}, to prove a relation between the source-to-solution maps given by \eqref{ss-map-eps} and the source-to-solution maps given by \eqref{ss-map-eps-tilde}, as well as a relation between the maps defined by \eqref{ss-map-ML} and \eqref{ss-map-M}. These relations are proved in Lemmas \ref{ss-map-coord} and \ref{ss-map-coord-0} respectively. With this, Theorem \ref{main-theorem} may be proven after we prove the following theorem.
\begin{theorem}\label{main-theorem-modified}
Let $V\subset M\setminus L$ be a relatively compact open set, $f\in L^{2}(V,g)$, and suppose that $k^{2}\in \CC$ is not an eigenvalue of the Laplace operator $-\Delta_{g}$ on $(M,g)$. Then, 
\begin{align}
\lim_{\eps\to 0}\Lambda_{V,\eps}f = \Lambda_{V,0}f \label{ss-map-convergence}
\end{align}
in the norm topology 
of $L^{2}(V,g)$.

Moreover, if additionally 
$m\ge0$ and  $1>\alpha>0$ satisfy $m+\alpha>1/2$, and
$f\in C^{m,\alpha}_{0}(V,g)$, then
\[ 
\lim_{\eps\to 0}\Lambda_{V,\eps}f = \Lambda_{V,0}f
\]
in $C^{m+2, \alpha}(\overline{V},g)$, where $\overline{V}$ is the closure of $V$ in $M$.
\end{theorem}


To prove Theorem \ref{main-theorem}, we first prove Theorem \ref{main-theorem-modified}. In this section, the proof of Theorem \ref{main-theorem-modified} will be split into several preliminary steps, which we group into subsections. These steps are outlined next.  

In Section \ref{function-spaces}, we first use the map $\Psi$ to define and establish equivalences between Sobolev spaces on $(M,g)$, $(M\setminus L,g)$, and $(\tM, \tg)$. 

Next, in Section \ref{fried-ext}, we define the Friedrich's extensions of operators in \eqref{helmholtz-eps} and \eqref{helmholtz} to an appropriate $L^2$-space. We then prove the claimed unique solvability of the problems \eqref{helmholtz2} and \eqref{helmholtz-eps}. Combined with the results in Section \ref{function-spaces}, we will then show that Theorem \ref{main-theorem} follows from Theorem  \ref{main-theorem-modified}.

Sections \ref{gamma-section} and \ref{resolvent-section} contain the bulk of the preliminary work needed to establish the limit \eqref{ss-map-convergence}. In Section \ref{gamma-section}, we use the machinery of $\Gamma$-convergence to prove \eqref{ss-map-convergence} for $k^2>0$. In Section \ref{resolvent-section}, we use classical linear perturbation theory to show \eqref{ss-map-convergence} for other values of $k\in \CC$, in particular for $k^2<0$, which correspond to sinusoidal solutions of \eqref{helmholtz-eps} and \eqref{helmholtz}.

Before proceeding, we define convenient notation which will be used from this point onward. We use $\nabla$ for the covariant derivative associated to $g$, and similarly write $\widetilde{\nabla}$ for the covariant derivative associated to $\tg$. Throughout this paper, $\mu$ denotes the Lebesgue measure on $\RR^3$, and $\mu_g$ (respectively $\mu_{\tg}$) denotes the measure induced on $M$ (respectively $\tM$) by $g$ (respectively $\tg$). We note here that given local coordinates $M$ or $\tM$, the volume forms satisfy $d\mu_g = \sqrt{|\det g|}d\mu$ and $\mu_{\tg} = \sqrt{|\det \tg|}d\mu$. We will abuse notation and write $d$ to denote exterior differentiation of forms either on $M$, $\tM$, or $\RR^3$, and use the context to make the distinction.

%
%
\subsection{Function spaces}\label{function-spaces}

Let $\Omega$ be a $C^\infty$-smooth 3-manifold. Here, the manifold $\Omega$ is either an open manifold (which boundary is not contained in $\Omega$), or a closed manifold. For any smooth Riemannian metric $h$ on $\Omega$, we write $L^1_{\text{loc}}(\Omega,h)$ for the set of all locally $d\mu_h$-integrable functions $u:\Omega\to \CC$.

We denote the space of square integrable functions on $(\Omega,h)$ by
\begin{align*}
L^2(\Omega, h) &= \left\{ u\in L^1_{\text{loc}}(\Omega,h)\,:\, \|u\|_{L^2(\Omega, h)} <\infty\right\},
\end{align*}
where the norm is given by
\[
\|u\|_{L^2(\Omega, h)}^2  = \int_\Omega |u|^2\,d\mu_h.
\]
When $h=g_e$, we use the shorthand notion $L^2(\Omega)=L^2(\Omega, g_e)$. Sometimes, it will be convenient to write $H^0(\Omega,h) = L^2(\Omega,h)$. Also, we observe that if $\sqrt{|\det h|}$ is bounded and measurable with respect to $\mu$, then $L^2(\Omega)\subset L^2(\Omega,h)$. 

Now, given a coordinate chart $(x^1,x^2,x^3): W\subset \Omega\to \RR$ and a multiindex $\alpha = (\alpha_1,\alpha_2,\alpha_3)$, we use $\partial^\alpha := \partial _{x^1}^{\alpha_1}\partial _{x^2}^{\alpha_2}\partial _{x^3}^{\alpha_3}$ to denote the distributional derivative. With this notation set, we next define the following Sobolev spaces. First,   
\begin{align*}
H^1(\Omega, h) &= \left\{ u\in L^1_{\text{loc}}(\Omega,h)\,:\, \forall |\alpha|\le1,\, \partial^\alpha u\in L^1_{\text{loc}}(\Omega,h),\,\|u\|_{H^1(\Omega, h)} <\infty\right\},
\end{align*}
where
\[
\|u\|_{H^1(\Omega, h)}^2 = \|u\|_{L^2(\Omega, h)}^2 + \int_\Omega h(\nabla u,\overline{\nabla u})\,d\mu_h,
\]
$\overline{\nabla u}$ is the complex conjugate of $\nabla u$, and $(\nabla u)^a = \sum_{b=1}^3 h^{ab}\partial_bu$ is the local coordinate expression of the covariant derivative of $u$. We emphasize that $\partial_b u$, $b=1,2,3$, denotes the distributional derivative, when it exists. Second, if $\Omega_1\subset \Omega$ is an open subset,
we define the space $H^1_0(\Omega_1, h)\subset H^1(\Omega, h)$, that is the closure of the smooth  functions that 
are compactly supported in $\Omega_1$,
\[
H^1_0(\Omega_1, h) := \text{cl}_{H^1(\Omega,h)}[C^\infty_0(\Omega_1)].
\]

Third,
\begin{align*}
H^2(\Omega, h) &= \left\{ u\in L^1_{\text{loc}}(\Omega,h)\,:\, \forall |\alpha|\le2,\, \partial^\alpha u\in L^1_{\text{loc}}(\Omega,h),\,\|u\|_{H^2(\Omega, h)} <\infty\right\},
\end{align*}
with norm
\[
\|u\|_{H^2(\Omega, h)}^2 = \|u\|_{H^1(\Omega, h)}^2 + \int_\Omega \|\nabla^2 u\|^2_h\,d\mu_h.
\]
Here, in coordinates $(x^1,x^2,x^3)$ on $\Omega$, we have 
\[
\|\nabla^2 u\|^2_h = \sum_{a,b,c,d=1}^3h^{ab}h^{cd}(\nabla^2 u)_{ac}(\overline{\nabla^2 u})_{bd},
\]
and $\nabla^2 u$ is the Hessian of $u$; it has the local expression
\[
(\nabla^2 u)_{ab} = \partial_a\partial_bu - \sum_{c,d=1}^3\frac12 h^{cd}(\partial_ah_{bd}+\partial_bh_{ad} - \partial_dh_{ab})\partial_cu.
\]
The Hessian is also related to the Laplace-Belrami operator on $\Omega$ through the trace over $h$:
\[
\Delta_hu = \text{tr}_h\nabla^2u.
\]
In local coordinates this becomes
$$\Delta_h u=\sum_{a,b =1}^3\frac{1}{\sqrt{|\det h|}}\partial_{a}\left(\sqrt{|\det h|} h^{ab}\partial_{b}u\right).$$

\begin{lemma}\label{restriction-iso}
Let $p=0,1,2$. Then, the restriction map
\[
\mathfrak{B}: H^p(M,g) \to H^p(M\setminus L,g), \ \ \mathfrak{B}(u) = u|_{M\setminus L} 
\]
is an isomorphism. Moreover, when we consider $M\setminus L$
as a subset $M$, the space  $H^1_0(M\setminus L, g) := \text{cl}_{H^1(M,g)}[C^\infty_0(M\setminus L)]$ satisfies
\[
H^1(M,g) =H^1_0(M\setminus L, g).
\]

\end{lemma}

\begin{proof}
First consider the case when $p=0$. Let $u\in L^2(M,g)$. As $L$ has measure zero with respect to the measure induced on $M$ by $g$, we have
\begin{align}
\|u\|_{L^2(M,g)}^2 
=\|\mathfrak{B}(u)\|_{L^2(M\setminus L,g)}^2.
\end{align}
From this, we see that $\mathfrak{B}$ is well-defined, continuous, injective, and an isometry onto its image. If $v\in L^2(M\setminus L, g)$, by extending $v$ to be zero on $L$, we obtain a function $u^v\in L^2(M,g)$ with $\mathfrak{B}(u^v)= v$. This proves that $\mathfrak{B}$ is surjective, and hence 
$\mathfrak{B}: L^2(M,g) \to L^2(M\setminus L,g)$ is an isomorphism.

Next, let $p=1$. To prove the claim in this case, we need only to show that the link $L\subset M$ has zero capacitance. Then, by \cite[Theorem 2.44]{KKM-sobolev-book} we achieve the desired isometry. 

For $\epsilon>0$, let $T(\eps)$ be a tubular neighbourhood about the link $ L\subset M$, as defined in \eqref{TubularL}, and let $r_0>0$
be such that the tubular neighborhood of radius $r_0$ is well defined.
The 2-capacitance \cite[Section 2.35]{KKM-sobolev-book} of $L$ in $T(r_0)$ is defined as 
\begin{align*}
 \text{Cap}(L) = \inf \{&\|\nabla u\|_{L^2(M)}^2\,:\, u\in H^1(M,g),\, u\ge1\text{ on a neighbourhood of $L$,}\\
 &\quad \text{and supp}(u)\subset \bar T(r_0)\}.   
\end{align*}

As in \cite[Example 2.12]{KKM-sobolev-book} and 
\cite[Thm. 2.2]{KKM-sobolev-book} (see also \cite[Cor 2.39]{KKM-sobolev-book}),
we see that the 2-capacity of the set $L$ in $T(r_0)$ is zero.



Before proceeding with the $p=2$ case, we prove the moreover part of the lemma. Since $L$ has vanishing capacity,  \cite[Theorem 2.45]{KKM-sobolev-book} (see also \cite{KKM-sobolev0}) implies that
\[
H^1_0(M\setminus L,g) = \text{cl}_{H^1(M,g)}[C^\infty_0(M\setminus L)] = H^1(M,g).
\]
Additionally, since $L$ has vanishing capacity,  by
\cite[Theorem 2.44]{KKM-sobolev-book}, we deduce that the map
\[
\mathfrak{B}: H^1(M,g) \to H^1(M\setminus L,g)
\]
is an isometric bijection.

Now suppose that $p=2$. As in the previous cases, it is easy to see that $\mathfrak{B}: H^2(M,g) \to H^2(M\setminus L,g)$ is well-defined, continuous, injective, and an isometry onto its image.

Let $v \in H^2(M\setminus L,g)$; by definition we also have $v \in H^1(M\setminus L,g)$. From the proof for the $p=1$ case, there exists a $u^v\in H^1(M,g)$ with \[
\mathfrak{B}(u^v)=v \]
and 
\[
\|u^v\|_{H^1(M,g)}=\|v\|_{H^1(M\setminus L,g)}. \]  

Let $W\subset M$ be a coordinate neighbourhood and $\psi \in C^\infty_0(W)$. As $H^1_0(M\setminus L,g) = H^1(M,g)$ and $\mathfrak{B}: H^1(M,g) \to H^1(M\setminus L,g)$ is an isomorphism, there exists functions $\phi_j\in C^\infty_0(W\setminus L)$ such that 
\[
\|\phi_j - \psi\|_{H^1(W,g)} \to 0
\]
as $j\to \infty$. Additionally, set $f=\Delta_gu\in L^2(M\setminus L,g)$. From the $p=0$ case proved above, there exists an $h^f\in L^2(M,g)$ with $\mathfrak{B}(h^f)=f$. We thus compute
\begin{align*}
\int_M g(\nabla u^v,\overline{\nabla \phi_j})\, d\mu_g &= \int_{M\setminus L} g(\nabla v,\overline{\nabla \phi_j})\, d\mu_g\\
& = - \int_{M\setminus L} f\bar{\phi}_j \, d\mu_g \\
&= -\int_{M} h^f\bar{\phi}_j \, d\mu_g.
\end{align*}
Taking the limit as $j\to \infty$ and integrating by parts gives us
\[
\int_M \Delta_g u^v \bar{\psi}\, d\mu_g =  \int_{M} h^f\bar{\psi} \, d\mu_g.
\]
Since both $\psi$ and the chart $W$ were arbitrary, we deduce that $u^v\in H^2(M,g)$. As $v \in H^2(M\setminus L,g)$ was arbitrary, this proves that $\mathfrak{B}:H^2(M,g) \to H^2(M\setminus L,g)$ is surjective, which is the last property we needed to confirm that $\mathfrak{B}:H^2(M,g) \to H^2(M\setminus L,g)$ is an isomorphism.

\end{proof}

\begin{lemma}\label{Psi-unitary}
For each $p\ge0$, the diffeomorphism $\Psi: M\setminus L\to \tM$ given by \eqref{transformation-map} induces a unitary operator
\begin{align*}
 \mathfrak{H}:H^p(\tM,\tg) \to H^p(M\setminus L,g), \quad \mathfrak{H}(\tu) = \tu\circ\Psi.   
\end{align*}

\end{lemma}

\begin{proof}
Let us consider first the case $p=0$. 

Let $\tu \in L^2(\tM,\tg)$ and set $u:= \mathfrak{L}(\tu)$. For $\tx\in \tM$ we also write $\tx$ for its local coordinate expression.  Since $\Psi$ is a diffeomorphism and $\tg = (\Psi^{-1})^*g$, employing a change of coordinates $x=\Psi^{-1}(\tx)$ gives
\begin{align}
\|\tilde u\|_{L^2(\tM,\tg)}^2 
&= \int_{M\setminus L} |\tu\circ \Psi(x)|^2\sqrt{|\det g(x)|}\left|\text{det}\,D\Psi(x)\right|\,d\tx\notag\\
&=\|u\|_{L^2(M\setminus L,g)}^2.\label{psi-isometry}
\end{align}
By construction, $\mathfrak{H}$ is a linear operator. From the above computation we achieve that $\mathfrak{H}$ is well-defined, continuous, and an isometry onto its image. To prove $\mathfrak{H}$ is unitary, it only remains to show that $\mathfrak{H}$ is surjective.

Let $u\in L^2(M\setminus L,g)$. Define $\tu:= u\circ \Psi^{-1}$. From \eqref{psi-isometry}, $\tu \in L^2(\tM,\tg)$ and we have $\mathfrak{H}(\tu) =u$. This concludes the proof for $p=0$.

The proof for $p>0$ follows similarly by writing formula
\eqref{psi-isometry} for the derivatives $\nabla u,\nabla^2 u,\dots$, $\nabla^p u$ and recalling that since $\tg = (\Psi^{-1})^*g$, the covariant derivatives obey $(\Psi^{-1})^*({\nabla})(w\circ \Psi^{-1}) = \nabla w$ for $w:M\setminus L\to \RR$ and $(\Psi^{-1})^*({\nabla})_{\Psi_*X}(\Psi_*Y) = \nabla_XY$ for vector fields $X,Y\in T(M\setminus L)$. Hence, for all  $u\in H^p(M\setminus L,g)$ we have 
$\tu =u\circ \Psi^{-1} \in H^p(\tM,\tg)$. This completes the argument for the surjectivity of $\mathfrak{H}$ in the case when $p>0$.
\end{proof}

We will also require use of H\"older function spaces. As previously,  let $(\Omega,h)$ be an open or closed $C^\infty$-smooth Riemannian 3-manifold. Given $m\ge0$ and $1>\alpha>0$, we define the $(m,\alpha)$-H\"older space as
\[
C^{m,\alpha}(\Omega,h)=\{ u\in C^m(\Omega,h)\,:\, \|u\|_{C^{m,\alpha}(K,h)}<\infty\hbox{ for all compact $K\subset \Omega$} \}
\]
and 
\[
C^{m,\alpha}_0(\Omega,h) =\{f\in C^{m,\alpha}(\Omega,h):\text{supp}(f) \text{ is a compact subset of $\Omega$}\}.
\]
In local coordinates $x$ on $K\subset \Omega$, the norm $\|u\|_{C^{m,\alpha}(K,h)}$ can be expressed as \[
\|u\|_{C^{m,\alpha}(K,h)} = \sup_{|\beta|\le m, x\in K}\|\nabla^\beta u(x)\|_h + \max_{|\beta|=m}\sup_{x\ne y\in K}\frac{\|\nabla^\beta u(x) - \nabla^\beta u(y)\|_{g_e}}{\text{dist}_h(x,y)^\alpha}
\]
where $g_e$ is the Euclidean metric on the coordinate chart and $\beta$ is a multiindex.

%
%

\subsection{Sesquilinear forms and their Friedrichs extensions}\label{fried-ext}

In this section, we will construct the elliptic operators appearing in equation \eqref{helmholtz} and the Neumann boundary value problem \eqref{helmholtz-eps} by using Friedrich's method. We then will relate these operators which act on functions defined on virtual space to operators which act on functions on physical space. 

We start with prescribing a set of forms which will eventually give rise to our desired operators in \eqref{helmholtz} and \eqref{helmholtz-eps}. Let $H$ be a Hilbert space and let
$\CC\cup\{\infty\}$ be the union of the complex plane and one point infinity. We consider
the map  
$\mathfrak{s}[\cdot, \cdot]: H\times H\to \CC\cup\{\infty\}$,
for which we define the notation
 $\mathfrak{s}[u]=\mathfrak{s}[u,u]$, 
 and 
the \emph{domain} of $\mathfrak{s}[\cdot, \cdot]$ (respectively $\mathfrak{s}[\cdot]$)  that is the subset $D\subset H$ such that 
 \begin{align*}
     \mathfrak{s}[\cdot,\cdot]: D\times D \to \CC \text{ and }
     \mathfrak{s}[\cdot]: D\to \CC
 \end{align*}
 have finite values. In this case, we write $\text{dom}(\mathfrak{s})=D$.

We call a map $\mathfrak{s}[\cdot, \cdot]: H\times H\to \CC\cup\{\infty\}$ an (unbounded) \emph{sesquilinear form} if for $v\in\text{dom}(\mathfrak{s})$ the map $\mathfrak{s}[\cdot, v]$ is linear and $\mathfrak{s}[v,\cdot]$ is conjugate-linear in $\text{dom}(\mathfrak{s})$. Given the form $\mathfrak{s}[\cdot,\cdot]$, we define the associated (unbounded) \emph{quadratic form} $\mathfrak{s}[\cdot]: H\to \CC\cup\{\infty\}$ by $\mathfrak{s}[u]=\mathfrak{s}[u,u]$.

Next consider the Hilbert space $L^{2}(M ,g)$ and designate by 
\begin{align*}
&\AAA_{0}: L^{2}(M ,g)\times L^{2}(M,g)\to \CC\cup\{\infty\}
\end{align*}
the sesquilinear form 
\[
\AAA_0[u,v]=\int_{M} g(\nabla u,\overline{\nabla v})\,d\mu_g
\]
defined on the domain 
 \[
 \text{dom}(\AAA_0) = H^1(M,g).
 \]

Second, let 
\begin{align*}
&\QQ_{0}: L^{2}(M\setminus L ,g)\times L^{2}(M\setminus L ,g)\to \CC\cup\{\infty\}
 \end{align*}
 be defined by
 \[
\QQ_{0}[u,v]= \int_{M\setminus L } g(\nabla u,\overline{\nabla v})\,d\mu_g
 \]
on the domain $\text{dom}(\QQ_0) = H^1(M\setminus L,g)$. 

Third and last, we define a 1-parameter family of forms which will be associated to the operators appearing in \ref{helmholtz-eps}. Let $T(\eps)$, with $\eps>0$, be a tubular neighbourhood of $L$ as in \eqref{TubularL}, and define a sesquilinear form $\QQ_{\eps}$ as
\begin{align*}
&\QQ_{\eps}: L^{2}(M\setminus L ,g)\times L^{2}(M\setminus L ,g)\to \CC\cup\{\infty\}
\end{align*}
defined by
\[
\QQ_{\eps}[u,v] = \int_{M \setminus T(\eps)}  g(du,d \bar{v})\,d\mu_g
\]
on the domain $$\text{dom}(\QQ_\eps) =\{ u \in L^{2}(M\setminus L ,g)\,:\, u|_{M\setminus T(\eps)}\in H^1(M\setminus T(\eps),g)\}.$$


Foremost, we record some properties of the forms $\QQ_{\eps}$ in the following lemma. 

\medskip

%
%
%

\begin{lemma}\label{form-properties}
Let $\AAA_0$, $\QQ_{0}$, and $\QQ_{\eps}$, $\eps>0$, be the forms defined above. Then,
 \begin{enumerate}
     \item the form $\AAA_0$ is densely defined in $L^2(M,g)$ and the forms $\QQ_{0}$ and $\QQ_{\eps}$, $\eps>0$, are densely defined in $L^2(M\setminus L,g)$.
 \end{enumerate}
 Additionally, these forms are
  \begin{enumerate}
    \setcounter{enumi}{1}
     \item closed,
     \item positive, and
     \item symmetric.
 \end{enumerate}
\end{lemma}

\begin{proof}
1.  By definition, $H^1(M\setminus L,g) \subset \text{dom}(\QQ_{\eps})$ and $H^1(M\setminus L,g) = \text{dom}(\QQ_{0})$. As $H^1(M\setminus L,g)$ is dense in $L^2(M\setminus L,g)$ and $\text{dom}(\AAA_0)=H^1(M,g)$ is dense in $L^2(M,g)$, this proves the first claim.

2. We now prove that $\QQ_{\eps}$ is closed for $\eps>0$; the proof is similar for $\QQ_{0}$ and $\AAA_0$.

 Let $v\in \text{dom}(\QQ_\eps)$; notice that
\begin{align*}
\QQ_{\eps}[v,v] \le \|v|_{M\setminus T(\eps)}\|_{H^{1}(M\setminus T(\eps) ,g)}^2.
\end{align*}

Suppose now that $u_{j} \in \text{dom}(\QQ_\eps)$ is a sequence such that $u_{j}\to u$ in $L^2(M\setminus L,g)$ and $\QQ_\eps[u_j,u_j]\to \QQ_\eps[u,u]$ as $j\to \infty$. We compute
\begin{align*}
|\QQ_{\eps}[u_{j}, u_{j}] -  \QQ_{\eps}[u,u]| &\le \left| \|u_{j}\|_{H^{1}(M\setminus T(\eps) ,g)}^2 - \|u\|_{H^{1}(M\setminus T(\eps) ,g)}^2 \right|\\
&=\|u_{j}\|_{H^{1}(M\setminus T(\eps) ,g)}\cdot \left(\sqrt{\QQ_\eps[u_j,u_j]} - \sqrt{\QQ_\eps[u,u]}\right)\\
&\quad+\|u\|_{H^{1}(M\setminus T(\eps) ,g)}\cdot \left(\sqrt{\QQ_\eps[u_j,u_j]} - \sqrt{\QQ_\eps[u,u]}\right).
\end{align*}
 As $\QQ_\eps[u_j,u_j]\to \QQ_\eps[u,u]$ and  $u_{j}\to u$ in $L^2(M\setminus L,g)$  as $j\to \infty$, we must have that $\|u_{j}\|_{H^{1}(M\setminus T(\eps) ,g)}+\|u\|_{H^{1}(M\setminus T(\eps) ,g)}$ is bounded. Thus, $u\in \text{dom}(\QQ_{\eps})$ which implies that $\QQ_{\eps}$ is closed.

3. and 4. The forms $\AAA_0$ and $\QQ_\eps$, $\eps\ge0$ are symmetric and positive as the tensor $\sqrt{|\det g|}g$ is both symmetric and positive.

\end{proof}

Next, we present a classical result built on the ideas of Friedrichs and Kato which demonstrates that the forms $\AAA_0$ and  $\QQ_\eps$, $\eps\ge0$, define self-adjoint operators on an appropriate $L^2$-space, and these operators extend the Laplace-Beltrami operator $\Delta_g$ to a subset of $L^2$-functions. This result holds due to the preceding Lemma \ref{form-properties}.

\begin{theorem}[Friedrich's Extension \cite{edmunds-evans}]\label{freidrichs-thm}
Let $\AAA_0$, $\QQ_\eps$, and $\QQ_\eps$, $\eps>0$, be the sesquilinear forms defined above. Then, there exists densely defined, closed, self-adjoint, linear operators
\begin{align*}
\mathcal{A}_{0}&:L^{2} (M, g)\to L^{2}(M ,g),\\
\HH_{0}&:L^{2} (M\setminus L, g)\to L^{2}(M\setminus L ,g),
\intertext{and}
\HH_{\eps}&:L^{2} (M\setminus L, g)\to L^{2}(M\setminus L ,g),\, \eps> 0, 
\end{align*}
which satisfy
\begin{enumerate} 
   \item  $\text{dom}(\mathcal{A}_0)= H^2(M)$;
  \item  $\text{dom}(\HH_\eps)=\{u\in H^1(M\setminus L): u|_{M\setminus T(\eps)}\in H^2(M\setminus T(\eps)), \partial_{\nu_\eps} u|_{\partial T(\eps)}=0\}$, $\eps>0$;
   \item $\text{dom}(\HH_0)=H^2(M\setminus L)$;
   \item $\AAA_0[u,v] = \langle \mathcal{A}_0 u,  v\rangle_{L^2(M,g)}$;
   \item $\QQ_0[u,v] = \langle \HH_0 u,  v\rangle_{L^2(M\setminus L,g)}$;
   \item $\QQ_\eps[u,v] = \langle \HH_\eps u,  v\rangle_{L^2(M\setminus L,g)}$, $\eps>0$;
   \item for $u\in \text{dom}(\HH_\eps)$, $\eps\ge0$, $\Delta_gu=-\HH_\eps u$ in the distributional sense on $M\setminus {T(\eps)}$ if $\eps>0$ and on $M\setminus L$ if $\eps=0$;
   \item for $u\in \text{dom}(\mathcal{A}_0)$,  $\Delta_gu=-\mathcal{A}_0 u$ in the distributional sense on $M$.
\end{enumerate}

\end{theorem}

With the above operators defined, we conclude this section by demonstrating  Theorem \ref{main-theorem} and Theorem \ref{main-theorem-modified}. To do so, we first prove that the problems \eqref{helmholtz2} and \eqref{helmholtz-eps} are uniquely solvable, and thus the source-to-solution maps given by \eqref{ss-map-ML} and \eqref{ss-map-eps} are well defined. The following lemma relates solutions to the Helmholtz equation on $(M,g)$ and the modified space $(M\setminus L,g)$.

\begin{lemma}\label{helmholtzM-tM-part1}
Let $k\in \CC$, $f\in L^{2}(M,g)$ and let $f_1=f|_{M\setminus L}$.

\begin{enumerate}
    \item Then $u_1\in H^{1}(M\setminus L,g)$ solves $\Delta_{g}u_1+k^{2}u_1=f_1$ on $M\setminus L$ in sense of distributions $\iff$ there is $u\in H^{1}(M,g)$ such that $u|_{M\setminus L}=u_1$ and $\Delta_{g}u+k^{2}u=f$ on $M$ in sense of distributions.
    \item If $u\in H^{1}(M,g)$ satisfies $\Delta_{g}u+k^{2}u=f$ on $M$ in sense of distributions
and $f\in L^{2}(M,g)$, then $u\in H^{2}(M,g)=\text{dom}(\mathcal A_0)$. 
\end{enumerate}

\end{lemma}

\begin{proof}
1. ``$\impliedby$'': Suppose that $u\in H^{1}(M,g)$ and $\Delta_{g}u+k^{2}u=f$ on $M$. By Lemma \ref{restriction-iso}, the restriction map $\mathfrak{B}:u\to u|_{M\setminus L}$ is a continuous map $\mathfrak{B}: H^{1}(M,g)\to H^{1}(M\setminus L,g)$. Hence, we can set $u_{1}=u|_{M\setminus L}$ and take the restriction of $\Delta_{g}u+k^{2}u=f$ to $M\setminus L$.

``$\implies$'': Assume that $u_1\in H^{1}(M\setminus L,g)$ solves $\Delta_{g}u_1+k^{2}u_1=f_1$ on $M\setminus L$. By Lemma \ref{restriction-iso}, the restriction map 
$$\mathfrak{B}: H^{1}(M,g)\to H^{1}(M\setminus L,g) \ \  \mathfrak{B}(u)= u|_{M\setminus L}$$
is an isomorphism. Thus, there is a distribution $u\in H^{1}(M,g)$ such that $u_{1}= \mathfrak{B}(u)$ and $\Delta_{g}u+k^{2}u=f$ on $M\setminus L$.

Next, let $H^{-1}(M,g)$ denote the dual space to $H^{1}_0(M,g)$ under the standard pairing. That is, $H^{-1}(M,g)$ is the set of all bounded, conjugate linear functionals $H^{1}_0(M,g)\to \CC$ (see \cite[Sections 1.1, 6.1.1]{edmunds-evans}, \cite[Section 5.9]{Evans-PDE}). Recall that $H^1_0(M\setminus L,g)$ is the closure of $C^\infty_0(M\setminus L)$ in $H^1(M,g)$ and that from Lemma \ref{restriction-iso} $H^1(M,g)=H^1_0(M\setminus L,g)$. 

Consider the distribution $b:=\Delta_{g}u+k^{2}u\in \mathcal D'(M)$.
Since $u\in H^{1}(M,g)$, we see that $b\in H^{-1}(M,g)$ and $\text{supp}(b)\subset L$. We compute
\begin{align}
\langle b, \phi\rangle_{ H^{1}(M,g)'\times  H^{1}(M,g)} &=\langle \Delta_{g}u+k^{2}u, \phi\rangle_{ H^{1}(M,g)'\times  H^{1}(M,g)}\notag \\
& = \int_{M\setminus L} ( \Delta_{g}u+k^{2}u)\bar{\phi}\,d\mu_{g} \notag\\
&=0\label{b-representation}
\end{align}
for all $\phi \in C^\infty_0(M\setminus L)$. Thus from Lemma \ref{restriction-iso} and Equation \eqref{b-representation} we see that 
\[
\langle b, v\rangle_{ H^{1}(M,g)'\times  H^{1}(M,g)}=0
\]
for all $v\in H^{1}(M,g)$; in other words $b=0$ in $H^{-1}(M,g)$.

Therefore, $\Delta_{g}u+k^{2}u=f$ in $L^{2}(M,g)$.



2. By writing $f$, $u$, and $\Delta_g$ in local coordinates, we may view $-\Delta_g$ as an elliptic operator on an open set of $\RR^3$. Then, we obtain via an application of \cite[Theorem 8.8]{GilbargTrudinger} the desired claim.

\end{proof}

\begin{lemma}\label{existence-helmholtz}
Let $V$ be an open and relatively compact set in $M\setminus L$ such that $\overline{V} \cap L =\emptyset$, and suppose that $k^2\in \CC$ is not an eigenvalue of $-\Delta_g$ on $M$ or 
 a Neumann eigenvalue of $\Delta_g$ on $M\setminus T(\eps)$. For $f\in L^2(M\setminus L, g)$ with support contained in $V$ we have: 
\begin{enumerate}
    \item for fixed $\eps>0$ and $f$ and $k^2$ as above, the problem \eqref{helmholtz-eps} has a unique solution $u_\eps\in H^1(M\setminus T(\eps))$, and the associated source-to-solution map \eqref{ss-map-eps} is well-defined.
    \item for $f$ and $k^2$ as above, the problem \eqref{helmholtz2} has a unique solution $u\in H^1(M\setminus L, g)$, and the source-to-solution map \eqref{ss-map-ML} is well-defined.
\end{enumerate}
\end{lemma}

\begin{proof}
1. We may view $-\Delta_g$ as an elliptic operator on the manifold $M\setminus T(\eps)$. Then, the claim follows by classical existence and uniqueness theorems for elliptic operators, for example \cite[Theorem 1.9]{edmunds-evans}.

2. First we note that as $(M,g)$ is compact, $H^1(M,g) = H^1_0(M,g)$. Then, by viewing $-\Delta_g$ as an elliptic operator on $(M,g)$, by classical existence and uniqueness theorems for elliptic operators, for example \cite[Theorem 1.3]{edmunds-evans}, the problem
\[
\Delta_gu+k^2u=f \text{ on } M
\]
has a unique solution in $H^1(M,g)$. Together with Part 1 of Lemma \ref{helmholtzM-tM-part1}, this implies that \eqref{helmholtz2} has a unique solution $u\in H^1(M\setminus L, g)$, as claimed. From this, we deduce that the source-to-solution map $\Lambda_{V,0}$ is well-defined.
\end{proof}

Next, we define finite solutions $\tu$ to the Helmholtz equation on $(\tM,\tg)$. We emphasize that the metric $\tg$ is singular on $\tM$. 

\begin{defn}\label{helmholtz-finite}
We say that $\tilde u\in H^1(\tM,\tg)$ is a \emph{finite energy solution} of the Helmholtz equation with source $\tilde f\in L^2(\tM,\tg)$ and wavenumber $k\in \CC$, $k\ne0$, 
if for all $\phi\in H^1(\tM,\tg)$
\begin{align}\label{finite-soln}
\int_{\tM} \left[\tg(\nabla\tilde u,\nabla \bar\phi) -k^2\tilde u\tilde\phi \right]\,d\mu_{\tg} &= -\int_{\tM} \tilde f\bar \phi \,d\mu_{\tg}.
\end{align}
When \eqref{finite-soln} holds, we write 
\begin{align}
\Delta_{\tg}\tu +k^2\tu &= \tilde{f} \quad \text{on } \tM.
\end{align}
\end{defn}

To be precise with the definitions, we recall
the standard definition that for $ u\in H^1(M\setminus L,g)$
the equation
$\Delta_{g}u +k^2u ={f}$ is valid on $M\setminus L$ in sense
of distributions, if for all $\phi\in C^\infty_0(M\setminus L)$,
\begin{align}\label{dist-soln}
\int_{M\setminus L} \left[g(\nabla u,\nabla \phi) -k^2 u\bar\phi \right]\,d\mu_{g} &= -\int_{M\setminus L} f\bar\phi \,d\mu_{g}
\end{align}
where we emphasize that the integrated functions on both sides are supported in a compact subset of $M\setminus L.$

\begin{lemma}\label{helmholtzM-tM-part2}
Let $\Psi$ be given by \eqref{transformation-map}, $k\in \CC$, $f\in L^{2}(M\setminus L,g)$ be compactly supported in $M\setminus L.$  Moreover, let $\tf=f\circ\Psi^{-1}\in
L^{2}(\tM,\tg)$. Then $u\in H^{1}(M\setminus L)$ solves $\Delta_{g}u+k^{2}u=f$ on $M\setminus L$ in sense of distributions $\iff$ $\tilde u=u\circ\Psi^{-1}\in H^{1}(\tM,\tg)$ is a finite energy solution of the equation $\Delta_{\tg}\tu+k^{2}\tu=\tf$ on $\tM$.
\end{lemma}

\begin{proof}

Suppose that $u\in H^{1}(M\setminus L)$ solves $\Delta_{g}u+k^{2}u=f$ on $M\setminus L$ in sense of distributions. 
We show next that $\tu=u\circ\Psi^{-1}$ is a finite energy solution of the Helmholtz equation on $\tM$.
First, by Lemma \ref{Psi-unitary}, we have that $\tu\in H^{1}(\tM)$.

Let $\epsilon>0$. To evaluate integrals to follow, let $T(\epsilon)$ be a tubular neighbourhood about the link $ L\subset M$, as defined in \eqref{TubularL}. Let $\tilde{T}(\epsilon)$ be a tubular neighbourhood about $\Sigma$, as defined in \eqref{TubularSigma}. 
Let $\chi_\eps\in L^\infty(\tM)$ be the indicator function of the set  $M\setminus \tilde{T}(\epsilon)$.

Given $\tilde\phi \in H^1(\tM,\tg)$, 
set $\phi=\tilde \phi\circ \Psi$. 
Then, $\phi\in H^1(M\setminus L)$. By Lemma \ref{restriction-iso}, $\phi$ and $f$ have extensions   $\phi_e\in H^1(M)$ and  $f_e\in L^2(M)$ that satisfy $\phi_e|_{M\setminus L}=\phi$ and $f_e|_{M\setminus L}=f.$ Since $u$ solves $\Delta_gu+k^2u=f$ on $M\setminus L$ in sense of distributions,
it follows from Lemma \ref{helmholtzM-tM-part1} (1)
that there is  $u_e\in H^1(M)$ that satisfies $u_e|_{M\setminus L}=u$ and
$\Delta_{g}u_e+k^{2}u_e=f_e$ on $M$, in sense of distributions.
As $f$ is supported in $V$, the 
essential support of $f_e$ does not intersect $L$, and we see that $u_e$ is $C^\infty$-smooth in some $M$-neighborhood of the set $L$. Moreover, $\Psi: M\setminus T(\epsilon)\to \tM\setminus \tT(\epsilon)$ is a Riemannian isometry, and thus
\begin{align}\label{integration by parts 1}
\int_{\tM} [\tilde g^{-1}(d\tilde u,d \bar{\tilde\phi}) -  k^2 \tilde u\bar{\tilde\phi}+\tilde f\bar{\tilde  \phi}]&\, d\mu_{\tg}\notag\\
&\hspace{-1cm}=  \lim_{\epsilon\to 0}\int_{\tM} \chi_\eps[\tilde g^{-1}(d\tilde u,d \bar{\tilde\phi}) -  k^2 \tilde u\bar{\tilde\phi} +\tilde f \bar{\tilde\phi}] \, d\mu_{\tg} \\ \label{integration by parts 2}
&\hspace{-1cm}= \lim_{\epsilon \to 0}\int_{M\setminus T(\epsilon)} \left[g^{-1}(d u,d \bar\phi) -  k^2  u\bar\phi  +f \bar\phi \right]\, d\mu_{g}\\
\label{integration by parts 3}
&\hspace{-1cm}=\int_{M\setminus L} \left[g^{-1}(d u,d \bar\phi) -  k^2  u\bar\phi  +f \bar\phi \right]\, d\mu_{g}\\
\nonumber
&\hspace{-1cm}=\int_{M} \left[g^{-1}(d u_e,d \bar\phi_e) -  k^2  u_e\bar\phi_e  +f_e \bar\phi_e \right]\, d\mu_{g}\\
&\hspace{-1cm}=0, \nonumber
\end{align}
as $L\subset M$ is zero-measurable and $u_e$ solves $\Delta_{g}u_e+k^{2}u_e=f_e$ on $M$. Therefore $\tilde u$ is a finite energy solution of the Helmholtz equation on $\tM$ for source $\tf$ and wavenumber $k$. 
%


The converse statement follows by observing that
when  $\tilde u$ is a finite energy solution of the Helmholtz equation on $\tM$, 
the formulas \eqref{integration by parts 1}-\eqref{integration by parts 3} hold
for all $\phi\in C^\infty_0(M\setminus L)$ and 
 $\tilde \phi=\phi\circ\Psi^{-1}$. Since the left hand side of formula \eqref{integration by parts 1} is equal to zero, we see that \eqref{integration by parts 3} is zero for all  $\phi\in C^\infty_0(M\setminus L)$. Thus  $u=\tu\circ \Psi$ solves the Helmholtz equation on $M\setminus L$ in sense of distributions.

\end{proof}

Now to show the relation of Theorem \ref{main-theorem} and Theorem \ref{main-theorem-modified}, we next use the diffeomorphism $\Psi:M\setminus L\to \tM$ to relate waves on $(\tM,\tg)$ and waves on $(M\setminus L,g)$.

\begin{lemma}\label{ss-map-coord}
Let $T(\eps) \subset M$ for $\eps>0$ be the tubular neighbourhood of $L$ defined in \eqref{TubularL}, and let $V\subset M\setminus L$ be open. We have for all $f\in L^{2}(M \setminus T(\eps),g)$
\begin{align}\label{cl1}
\Lambda_{V,\eps}(f) &=\Psi^{*} \widetilde{\Lambda}_{\Psi (V), \eps}(\Psi^{-1})^*f.
\end{align}
Additionally, for all $f\in L^{2}(M \setminus L,g)$ having a compact support in $M \setminus L$,
\begin{align}\label{cl2}
\Lambda_{V,0}(f) &=\Psi^{*} \widetilde{\Lambda}_{\Psi (V)}(\Psi^{-1})^*f.
\end{align}
\end{lemma}

\begin{proof} 
As for $\eps>0$ the metric tensor $\tg$  in
the set $\tM \setminus \tilde T(\eps)$ is bounded
from above and below by positive constants and the
sets $M \setminus T(\eps)$ and
$\tM \setminus \tilde T(\eps)$
are  $C^\infty$-smooth manifolds with
a $C^\infty$-smooth boundary, the first claim \eqref{cl1} follows by applying the standard change of variables by 
the $C^\infty$-smooth map $\Psi:M \setminus T(\eps)\to \tM \setminus \tilde T(\eps)$.

The second claim \eqref{cl2} follows by the Lemma \ref{Psi-unitary} and Lemma \ref{helmholtzM-tM-part2}
for sources $f\in
L^{2}(M\setminus L,g)$ and $\tf=(\Psi^{-1})^*f\in
L^{2}(\tM,\tg)$ and the corresponding solutions 
 $u\in
H^{1}(M\setminus L,g)$ and $\tu=(\Psi^{-1})^*u\in
H^{1}(\tM,\tg)$.
\end{proof}

To obtain the relation of Theorem \ref{main-theorem} and Theorem \ref{main-theorem-modified}, it now only remains to prove an equivalence between measurements on the virtual space $(M,g)$ and measurements on the modified space $(M\setminus L,g)$.

\begin{lemma}\label{ss-map-coord-0}
Let $V\subset M\setminus L$ be open. We have for all $f\in L^{2}(M,g)$
\[
\Lambda_{V,0}(\mathfrak{B}f) = \Lambda_{V}f,
\]
\end{lemma}
where  $\mathfrak{B}:L^2(M,g)\to L^2(M\setminus L,g)$  is the (bijective) restriction map.

\begin{proof}
After the above results, the proof follows easily but we present the details for clarity.





Let $f\in L^2(M,g)$ and let $u\in H^1(M,g)$ be the solution of $\Delta_gu+k^2u=f$ on $M$ in sense of distributions. In fact, then $u\in H^2(M,g)$.

We recall that $\mathfrak{B}:L^2(M,g)\to L^2(M\setminus L,g)$ and $\mathfrak{B}:H^1(M,g)\to H^1(M\setminus L,g)$ are continuous, and denote $f_1 = \mathfrak{B}(f)\in L^2(M\setminus L,g)$ and $u_1 = \mathfrak{B}(u)\in H^1(M\setminus L,g)$. 
Then $\Delta_gu_1+k^2u_1=f_1$ on $M\setminus L$ in sense of distributions and 
$$ 
\Lambda_{V,0}(\mathfrak{B}f)=u_1|_{V} =u|_{V}=
\Lambda_{V}(f).
$$
This proves the claim.

\end{proof}

Since now our aim is to prove that \eqref{ss-map-convergence} holds, our next natural step is to demonstrate that $\QQ_\eps$ converges to $\QQ_0$ in a suitable sense. This is the goal of the next Section.

%
%
\subsection{$\Gamma$-convergence of the sesquilinear forms}\label{gamma-section}

We recall the definition of $\Gamma$-convergence:

\begin{defn}[\cite{dalmaso} and \cite{braides-handbook}]
Let $(X,\tau)$ be a topological space and $\{\mathcal{F}_{\eps}:X\to [-\infty,\infty],\,\eps>0\}$ be a 1-parameter family of functionals on $X$. For $x\in X$ let $N(x)$ denote the set of all open neighbourhoods $U\subset X$ of $x$, with respect to the topology $\tau$. If 
\[
\mathcal{F}= \sup_{U\in N(x)}\liminf_{\eps\to0}\inf_{y\in U} \mathcal{F}_{\eps} = \sup_{U\in N(x)}\limsup_{\eps\to0}\inf_{y\in U} \mathcal{F}_{\eps},
\]
we say that \emph{$\mathcal{F}_\eps$ $\Gamma$-converges to $\mathcal{F}$} in $X$ with respect to the topology $\tau$ and write 
\[
\mathcal{F}_\eps \overset{\Gamma}{\rightharpoonup} \mathcal{F}.
\]

\end{defn}

In the context of cloaking, $\Gamma$-convergence has been used earlier, for example in \cite{FKR2014,GKLU-quantum}.

We now proceed to formulate the stability of solutions to the Helmholtz problems \eqref{helmholtz-eps}. This is done by considering the quadratic forms $\QQ_{\lambda}$ and $\QQ_{\lambda,\eps}$ associated to the Helmholtz operators $\HH_{\lambda}$ and $\HH_{\lambda,\eps}$ in the framework of $\Gamma$-convergence.  
\medskip


%
%
%

\begin{lemma}\label{lemma-q1}
The quadratic form
\begin{align*}
 \QQ^1_\eps[u] &= \QQ_\eps[u,u], u\in L^{2}(M\setminus L ,g)
\end{align*}
converges pointwisely in $L^2(M\setminus L,g)$ to
\[
\QQ^1[u]:= \QQ_0[u,u], u\in L^{2}(M\setminus L ,g).
\]
Additionally, 
\[
\QQ^1_\eps\overset{\Gamma}{\rightharpoonup}\QQ^1
\]
in the norm topology of $L^{2}(M\setminus L ,g)$ and in the weak topology of $H^{1}(M\setminus L ,g)$. 
\end{lemma}

\begin{proof}
First we show that $\QQ^1_\epsilon\to\QQ^1$ pointwisely in $L^{2}(M\setminus L ,g)$. Let $u\in L^{2}(M\setminus L ,g)$.

If $u\notin H^{1}(M\setminus L ,g)$, then $\QQ^1_\epsilon[u]=\infty = \QQ^1[\tu]$ for all $\epsilon>0$. 

Thus, suppose that $u\in H^{1}(M\setminus L ,g)$. 
Since the manifold $M$ is compact, there exists a finite collection of open sets $W_j\subset M$, $j=1,\dots, K$, which cover $M$: $M=W_1\cup W_2\cup \dots W_K$. Thus, we may express $u$ as a finite sum $\sum_{j=1}^K u_j$ where $u_j \in H^1(W_j\setminus L, g)$. Thus for simplicity of the argument that follows, we assume that $u$ is supported in a set $W\setminus L\subset M\setminus L$ and there are local coordinates $x:W\setminus L\to \RR^3$  on $M$. Next, we write $u(x)$, $x\in W\setminus L$, for the coordinate expression of $u$.


For $m\in \mathbb{N}$ and $\xi\in \CC^3$, let 
\[
w_m({x},\xi):= \chi_{1/m}({x})\sqrt{|\det g({x})|}{g}^{ab}({x})\xi_a\bar{\xi}_b,
\]
where $a,b\in\{1,2,3\}$ and 
$ \chi_{M\setminus  T(1/m)}$ is a cutoff function which vanishes on $T(1/m)$. For $\xi$ fixed, each $w_m(\cdot,\xi)$ is measurable with respect to the Euclidean Lebesgue measure $d\mu_e$ on the coordinate chart $x(W)\subset \RR^3$. Further, these functions satisfy 
\begin{align}\label{Q1-inc}
0\le w_m(\cdot,\xi)\le w_{m+1}(\cdot,\xi)
\end{align}
for all $m\in \mathbb{N}$. As $m\to\infty$,
\begin{align}\label{Q1-pointwise-M}
w_m(\cdot,\xi)\to w(\cdot,\xi):= \sqrt{|\det g|}g^{ab}(\cdot)\xi_a\bar{\xi}_b.
\end{align} 
Using the above monotone pointwise convergence of $\{w_m(\cdot,\xi)\}_{m\in\mathbb{N}}$ in $M\setminus L$ and Lebesgue monotone convergence theorem,
\begin{align*}
 \lim_{\epsilon\to0}\QQ^1_\epsilon[u] &= \lim_{m\to\infty} \int_{W\setminus L } w_m(x,\partial u({x}))\, d\mu_e(x)\\
 & = \int_{W\setminus L } {g}^{ab}({x})\partial_a u({x})\partial_b\bar{u}({x})\,d\mu_g(x) = \QQ^1[u].
 \end{align*}
This demonstrates the claimed pointwise convergence $\QQ^1_\epsilon\to\QQ^1$ as $\eps\to0$. From \eqref{Q1-inc}, we see $\QQ^{1}_{\eps}[u]\ge0$ for all $u\in L^{2}(M\setminus L ,g)$ and all $\eps\ge0$.

Without loss of generality, let $\eps=\frac1m$ for $m\in\mathbb{N}$ and consider the sequence $\{\QQ^1_m\}_{m\in\mathbb{N}}$. We will now prove that $\QQ^1_m$ $\Gamma$-converges to $\QQ^1$ as $m\to\infty$ strongly in $L^2(M\setminus L,g)$.

Since $\QQ^1_m\to \QQ^1$ pointwisely, if additionally $\{\QQ^1_m\}_{m\in\mathbb{N}}$ is an increasing sequence of lower semi-continuous functionals, then \cite[Proposition 5.4]{dalmaso} implies the desired $\Gamma$-convergence in the strong topology of $L^2(M\setminus L,g)$. 
 From \eqref{Q1-inc} we see that $\{\QQ^1_m\}_{m\in\mathbb{N}}$ is increasing. Thus to conclude the proof we next show that each $\QQ^1_m$ is a lower semi-continuous function on $L^{2}(M\setminus L ,g)$.
 
 Indeed, let $u\in L^2(M\setminus L,g)$ and let $(u_\ell)_{\ell\in \mathbb{N}}\subset L^2(M\setminus L,g)$ be a sequence converging to $u$ in $L^2(M\setminus L,g)$.  
 Suppose that $\liminf_{\ell\to \infty}\QQ_\eps^1(u_\ell)< \infty$; otherwise there is nothing to prove. By definition of the limit inferior, there exists a subsequence 
  $(u_{\ell_j})_{j\in \mathbb{N}}$ of $(u_\ell)_{\ell\in \mathbb{N}}$ such that the sequence $\|u_{\ell_j}|_{M\setminus T(\eps)}\|_{H^1(M\setminus T(\eps))}$ is uniformly bounded.
Again, by choosing a subsequence, we can assume that the functions  $u_{\ell_j}|_{M\setminus T(\eps)}$
converge weakly in $H^1(M\setminus T(\eps))$.
Then, by the weak lower semi-continuity of the norm in the weak topology of $H^1(M\setminus T(\eps),g)$, we see that
\[
\QQ_\eps^1[u]\le \liminf_{\ell\to \infty}\QQ_\eps^1[u_\ell].
 \]
This demonstrates that $\QQ^1_\eps$ is lower semicontinous in the strong topology of $L^2(M\setminus L,g)$ and completes the proof of $\Gamma$-convergence in the strong topology of $L^2(M\setminus L,g)$.

Now, as $g$ is a $C^\infty$-smooth Riemannian metric on $M$, and since we have the pointwise convergence \eqref{Q1-pointwise-M}, by \cite[Proposition 5.14]{dalmaso}, $\QQ^1_\eps$ $\Gamma$-converges to $\QQ^1$ in the weak topology of $H^1(M\setminus L,g)$.

\end{proof}

%
%
%

\begin{lemma}\label{lemma-q2}
For $\lambda \in(-\infty,0)$, and every $u\in L^{2}(M\setminus L ,g)$, the functional
\[
\QQ^2_{\lambda,\eps}[u]:=  \lambda\int_{M\setminus T(\eps) }|u|^2\,d\mu_g,
\]
converges pointwisely in $L^2(M\setminus L,g)$ to
\[
\QQ^2_\lambda[u]:=  \lambda\int_{M\setminus L }|u|^2\,d\mu_g.
\]
 Moreover,
\[
\QQ^2_{\lambda,\eps}[u]\overset{\Gamma}{\rightharpoonup} \QQ^2_{\lambda}[u]
\]
with respect to the norm topology of $L^2(M\setminus L ,g)$.
\end{lemma}

\begin{proof}
We first note that $\text{dom}(\QQ^2_{\lambda}) = L^2(M\setminus L ,g)$ and is lower semicontinuous on $L^2(M\setminus L ,g)$. In particular, it agrees with its lower semicontinuous envelope, $sc^- \QQ^2_{\lambda}$. 

Given that $\lambda <0$, if $\epsilon_{2}> \eps_1>0$, we have $ \QQ^2_{\lambda,\eps_2}> \QQ^2_{\lambda,\eps_1}$. Thus $\QQ^2_{\lambda,\eps}$ decreases as $\eps\to0$.

Let $u\in L^2(M\setminus L,g)$. Then, using the Lebesgue dominated convergence theorem we compute
\[
\left|\QQ^2_{\lambda,\eps}[u] - \QQ^2_{\lambda}[u]\right| = \left|-\lambda\int_{ T(\epsilon)}|u|^2 d\mu_{g}\right|\to 0.
\]
That is, $\QQ^2_{\lambda,\eps}\to \QQ^2_{\lambda}$ pointwisely in $L^2(M\setminus L,g)$. 

Then, by \cite[Proposition 5.7]{dalmaso}, $\QQ^2_{\lambda,\eps}$ $\Gamma$-converges to $sc^-( \QQ^2_{\lambda})= \QQ^2_{\lambda}$ in $L^{2}(M\setminus L ,g)$, as desired.

\end{proof}

\begin{defn}
Let $\eps\ge0$. The \emph{resolvent set} of $\HH_{\eps}$ , denoted by $\text{res}(\HH_{\eps})$, is the set of $\lambda\in \CC$ such that the associated \emph{resolvent operator} 
\[
\Res_\eps(\lambda): L^{2}(M\setminus L ,g)\to L^{2}(M\setminus L ,g), \ \ \Res_\eps(\lambda)\tf:= (\HH_{\eps} - \lambda)^{-1}\tf,
\]
is bounded and satisfies 
\[
(\HH_{\eps} - \lambda)\Res_\eps(\lambda) = \Res_\eps(\lambda)(\HH_{\eps} - \lambda) = \text{Id}_{L^{2}}.
\]


Then, the \emph{spectrum} of $\HH_{\eps}$ is the complement set $\text{spec}(\HH_{\eps}):= \CC\setminus \text{res}(\HH_{\eps})$.


\end{defn}

%
%
%
\medskip
\begin{corollary}\label{dal-maso-cor}
Let $\lambda\in \CC$. Then
\begin{enumerate}
\item $\QQ_{\eps}^1+\QQ^2_{\lambda,\eps} \to \QQ^1 +\QQ^2_{\lambda}$ pointwisely in $L^2(M\setminus L ,g)$.
\item $\QQ^1_{\eps} +\QQ^2_{\lambda,\eps} \overset{\Gamma}{\rightharpoonup} \QQ^1+\QQ^2_{\lambda}$ in the norm topology of $L^2(M\setminus L,g)$.
\item for $\lambda>0$, the operator $\mathcal{R}_\eps(\lambda)$  converges in the norm topology of $L^2(M\setminus L ,g)$ to the resolvent operator $\mathcal{R}_0(\lambda)$ associated to $\HH_0$.
\end{enumerate}
\end{corollary}

\begin{proof}
 First consider the case when $\lambda\le0$. From Lemmas \ref{lemma-q2} and \ref{lemma-q1}, $\QQ^1_\eps\overset{\Gamma}{\rightharpoonup} \QQ^1 $ and $\lambda\QQ^2_\eps\overset{\Gamma}{\rightharpoonup} \lambda\QQ^2 $.
 Further, we have the pointwise convergence
$ \QQ^1_\eps[u]\to \QQ^1[u]$ and $\QQ^2_{\lambda,\eps}[u]\to \QQ^2_\lambda[u]$
for each $u\in L^2(M\setminus L ,g)$. Then,  by \cite[Proposition 6.25]{dalmaso}, (1) and (2) hold in this case. 

Next consider $\lambda>0$. By Lemma \ref{form-properties}, the quadratic forms $\QQ_\eps^1+ \QQ_{\lambda,\eps}^2$ and $\QQ^1 + \QQ_\lambda^2$ are positive. In this setting we may apply \cite[Theorem 13.6]{dalmaso}  to achieve claims (1)-(3).

\end{proof}

%
%

\subsection{Strong resolvent convergence}\label{resolvent-section}

%
%
%

In the previous section, we concluded with Corallary \ref{dal-maso-cor} which gave strong convergence of the resolvent operators $\mathcal{R}_\eps(\lambda)$ to $\mathcal{R}_0(\lambda)$ as $\eps\to 0$ in the case when $\lambda>0$. In this section, we prove that as $\eps\to0$, the resolvents $\Res_{\eps}(\lambda)$ converge strongly in $L^{2}(M\setminus L ,g)$ to $\Res_0(\lambda)$, for an appropriate set of complex values $\lambda\in \CC$. In particular, we will show that the values of $\lambda$ for which we have strong convergence include those of the form $-k^{2}<0$ for some $k\in\RR$, which correspond to sinusoidal wave solutions of the Helmholtz equations $-\HH_{\eps}u+k^{2}u = f$. 

First, we show that a compact set $K\subset \CC$ which avoids the the spectrum of $\HH_0$ also avoids the spectrum of $\HH_{\eps}$ for sufficiently small $\eps>0$:\medskip

\begin{lemma}\label{spec-avoidance}
Let $K\subset \CC$ be compact and such that $K\cap\text{spec}(\HH_0)=\emptyset$. Then, there exists an $\eps_{K}>0$ such that for $\eps<\eps_{K}$ 
\[
K\cap \text{spec}(\HH_{\eps}) = \emptyset.
\]
\end{lemma}

\begin{proof}


By \cite[Chapter IV, Section 3.1, Theorem 3.1]{kato-PLO} and \cite[Chapter IV, Section 2.6, Theorem 2.25]{kato-PLO} is is enough to show that there exists some $\lambda_0\in \text{res}(\HH_0)$ and an $\eps_0>0$ such that for $\eps<\eps_0$,
\begin{align}\label{res-conv1}
    \| R_\eps(\lambda_0) - R_0(\lambda_0)\|_{L^2(M\setminus L,g) \to L^2(M\setminus L,g)} \to 0 \ \ \text{as } \eps\to 0.
\end{align}

 As $\HH_\eps$ are positive, self adjoint operators, if $\lambda_0>0$, then $\lambda_0 \in \text{res}(\HH_\eps)$ for all $\eps>0$. From Corollary \ref{dal-maso-cor}, we achieve \eqref{res-conv1} with $\lambda_0 =1$.

\end{proof}

Next, we introduce notation for subsets of $\CC$ for which the family of resolvents $\Res_{\eps}(\lambda)$ exhibit strong convergence or boundedness as $\eps\to 0$.

\begin{defn}\cite[Chapter VIII, Section 1.1]{kato-PLO}
The \emph{region of boundedness} for the family $\left\{\Res_{\eps}(\lambda)\,:\, \eps>0\right\}$ is the subset of all $\lambda\in\CC$ with the property that there exists an $\eps_0>0$ such that the family $\left\{\|\Res_{\eps}(\lambda)\|_{L^{2}}\,:\, \eps_0>\eps>0\right\}$ is bounded. We denote this set by $D_{b}$.

The set of values $\lambda\in \CC$ for which the strong convergence limit $\lim_{\eps\to0}\Res_{\eps}(\lambda)$ exists in $L^{2}(M\setminus L ,g)$ is denoted by $D_{s}$, and is called the \emph{region of strong convergence} for the family $\left\{\Res_{\eps}(\lambda)\,:\, \eps>0\right\}$.

Finally, let $D_{u}$ denote the set of $\lambda\in \CC$ for which the family of resolvents $\left\{\Res_{\eps}(\lambda)\,:\, \eps>0\right\}$ convergences in norm to some limit in $L^{2}(M\setminus L ,g)$. We call $D_{u}$ the \emph{region of convergence in norm}. Observe that $D_{u}\subset D_{s}\subset D_{b}$.
\end{defn}

With these definitions in hand, we prove:

\begin{lemma}\label{resolvent-convergence}
Let $\lambda\in \text{res}(\HH_0)$. For $f\in L^2(M\setminus L ,g)$, we have as $\eps \to 0$,
\[
\Res_\eps(\lambda)f \to \Res_0(\lambda)f \]
in the norm topology of $L^{2}(M\setminus L ,g)$. Furthermore, if $K\subset \CC$ is compact and satisfies $K\cap\text{spec}(\HH_0)=\emptyset$, the convergence is uniform for $\lambda \in K$.
\end{lemma}

\begin{proof}

The first part of the claim is to show that $D_s = \CC\setminus \text{spec}(\HH)$. We prove it by using an open/closed argument which gives $D_s=D_b = \CC\setminus \text{spec}(\HH)$.

To begin, observe that $D_{b}$ is connected since $\HH_{0}$ is self-adjoint and thus the spectrum of $\HH_0$ is both discrete and countable \cite[Chapter 6, Theorem 1.9]{edmunds-evans}. From Corollary \ref{dal-maso-cor} we know the sets $D_s$ and $D_b$ are nonempty since $\RR_{+}\subset D_{s}\subset D_b$.

Now, let $K\subset \CC$ be compact and such that $K\cap\text{spec}(\HH_0)=\emptyset$. From Lemma \ref{spec-avoidance} there is an $\eps_{K}$ and $ \delta>0$ such that the set
\[
K_{\delta} = \{ z\in \CC\,:\, \text{dist}(z,K)<\delta\}
\]
satisfies $K_{\delta}\cap \text{spec}(\HH_{\eps})= \emptyset$ for all $\eps<\eps_{K}$.  

 Given that the operators $\HH_{\eps}:L^{2}(M\setminus L ,g)\to L^{2}(M\setminus L ,g)$, $\eps\ge0$, are self-adjoint, by \cite[Chapter V, Section 3.5]{kato-PLO}, for all $\lambda\in K$ 
\begin{align}
\|\Res_{\eps}(\lambda)\|_{L^{2}\to L^{2}} \le [\text{dist}(\lambda, \text{spec}(\HH_{\eps})]^{-1} < \delta^{-1}.
\end{align}
Thus we deduce that $K\subset D_{b}$. Since $K\subset \CC\setminus \text{spec}(\HH_0)$ was arbitrary and $\CC\setminus \text{spec}(\HH_0)$ is connected, we have $D_{b}=\CC\setminus \text{spec}(\HH_0)$.

Utilizing \cite[Chapter VIII, Section 1.1, Theorem 1.2]{kato-PLO}, we have that $D_{s}$ is both relatively open and relatively closed in $D_{b}$. Thus $D_s=D_b=\CC\setminus \text{spec}(\HH_0)$. 

The result\cite[Chapter VIII, Section 1.1, Theorem 1.2]{kato-PLO} further states that for $f \in L^{2}(M\setminus L,g)$, $\Res_\eps(\lambda)\tf \to \Res_0(\lambda)f$ uniformly on all compact subsets of $D_b$, which completes our proof.

\end{proof}

%
%
%


\subsection{Proof of Theorem \ref{main-theorem}}\label{the-main-proof}



 Now we are ready to prove our main result, Theorem \ref{main-theorem}. As shown in Section \ref{fried-ext}, Theorem \ref{main-theorem} follows from Theorem \ref{main-theorem-modified}. We provide the proof of Theorem \ref{main-theorem-modified} below.
%
%
%
%

\begin{proof}[Proof of Theorem \ref{main-theorem-modified}]

Let $f\in L^{2}(M\setminus L,g)$ be supported in $V$ and
let $u_1\in H^{1}(M\setminus L,g)$ be a distribution which is a finite energy solution of 
\begin{align*}
\Delta_gu_1+k^2u_1=f \ \ \text{on } M\setminus L,
\end{align*}
in sense of Definition \ref{helmholtz-finite}. Then $u_1=-\Res_0(-k^2)f$.

Define $\tf =f\circ  \Psi^{-1}$ and $\tu =u_1\circ  \Psi^{-1}$
on $\tM$. By Lemma \ref{Psi-unitary}, $\tf \in L^{2}(\Psi(V),\tg)$. Moreover,
by Lemma \ref{helmholtzM-tM-part2},  $\tu \in H^{1}(\tM,\tg)$ is a finite energy solution of the Helmholtz equation on $\tM$ with wavenumber $k$ and source $\tf$. 

Denote by $u_{\eps} \in H^{1}(M\setminus T(\eps), g)$ the solution to \eqref{helmholtz-eps} with wavenumber $k$ and source $f$. Then, $u_\eps=-\Res_\eps(-k^2)f$ and
by Lemma \ref{resolvent-convergence}, 
\begin{align}\label{u-convergence-l2}
\lim_{\eps\to 0}u_{\eps} = u_1
\end{align}
with respect to the norm topology of $L^{2}(M\setminus L,g)$. This implies that
\[ 
\lim_{\eps\to 0}u_{\eps}|_{V} = u_1|_{V}
\]
with respect to the norm topology of $L^{2}(V,g)$. By definition we have $u_\eps|_{V}= \Lambda_{V,\eps}f$ and $u_1|_{V}= \Lambda_{V}f$; applying Lemma \ref{ss-map-coord} we thus obtain the first part of the claim.  

For the moreover part of the claim, suppose that $f\in C^{m,\alpha}_0(V,g)$ for some $m\ge0$ and $\alpha\in (0,1)$. Let $W\subset M\setminus L$ be a relatively compact open set having a smooth boundary such that $\overline V\subset \overline W$ and $\overline W \subset M\setminus L$; here the over-line denotes set closure in $M$. Then, we have for $w= u$ or $u_\eps$ that $w\in C^{m+2,\alpha}(\overline{V})$ and the standard elliptic estimate ( see \cite[Theorem A.2.3]{Jost} or \cite[Ch. 6 Ex 6.1]{GilbargTrudinger})
\begin{align}\label{estimate1}
\|w\|_{C^{k,\alpha}(\overline V)} \le C(\|f\|_{C^{k,\alpha}(\overline  W)} + \|w\|_{C^{0,\alpha}(\overline W)}),
\end{align}
holds for some $C= C(M\setminus L, \overline V, \overline W, g)>0$.  

There also exists a $C= C(M\setminus L, W, g)>0$ such that
\begin{align}\label{estimate2}
\| w\|_{H^{2}(W,g)}\le C\|w\|_{L^2(M\setminus L,g)}.
\end{align}
By applying Morrey's inequality for $\text{dim}(W)=3$ and $H^{2}(W,g)$, we have
\begin{align}\label{estimate3}
 \| w\|_{C^{0,1/2}(\overline W)}\le C\| w\|_{H^2(W,g)}   
\end{align}
for some $C=C(\overline W,g)>0$.

Since we assumed that $m+\alpha>1/2$, putting together estimates \eqref{estimate1} - \eqref{estimate3} gives us the inequality
\[
\| u_\eps - u\|_{C^{m,\alpha}(\overline V,g)}\le C\| u_\eps - u\|_{L^2(W,g)},
\]
for $C= C(M\setminus L, \overline V, \overline W, m,\alpha, g)>0$.

Using, \eqref{u-convergence-l2} with $M\setminus L$ replaced by $W$, we obtain the desired convergence in $C^{m,\alpha}(\overline{V},g)$.

\end{proof}

%
%

%
%

\section{Proof of Proposition \ref{smooth-embedding}}
\label{const-3manifold}

In this section, we prove a version of Proposition \ref{smooth-embedding} which also satisfies the necessary additional properties in \eqref{explicit diffeo}. The proof of the existence of the embedding is based on a characterization theorem for closed $3$-manifolds due to Lickorish \cite{Lickorish} and Wallace \cite{Wallace}: \emph{Each closed and oriented $3$-manifold is obtained by a surgery along a link in $\SSS^3$}. In particular, each closed and oriented $3$-manifold admits an embedding into $\RR^3$ after removal of a suitable link. 


We do not discuss the proof of this topological statement in detail, but introduce the necessary terminology related to a precise statement and recall a smoothing result in $3$-dimensions used to deduce Proposition \ref{smooth-embedding} from the Lickorish--Wallace theorem.

\newcommand{\interior}{\mathrm{int}}

Let $M$ be a $3$-manifold. We call the product space $\bar B^2 \times \SSS^1$ and the circle $\{0\} \times \SSS^1$ the \emph{solid $3$-torus} and its \emph{core curve}, respectively. Given an embedding $\phi \colon \bar B^2 \times \SSS^1 \to M$, we call the images $\phi(\bar B^2\times \SSS^1)$ and $\phi(\{0\} \times \SSS^1)$ an \emph{embedded solid $3$-torus} and its core curve, respectively.

Heuristically, a surgery over a circle $S$ on a $3$-manifold $M$ is an operation which replaces an interior of an embedded solid $3$-torus $T$ in $M$ by the interior of another solid $3$-torus $T'$ which is not a priori contained in $M$. Formally, the replacement is obtained by gluing a copy of $\bar B^2 \times \SSS^1$ to $M\setminus \interior T$. More precisely, we fix first a neighbourhood for $S$ by choosing an embedding $\phi \colon \bar B^2 \times \SSS^1 \to M$ for which $S=\phi(\{0\}\times \SSS^1)$. To formalize the gluing, let $h \colon (\partial \bar B^2)\times \SSS^1 \to \phi((\partial B^2)\times \SSS^1)$ be a homeomorphism. We consider then a disjoint union
\[
\left(M \setminus \phi(B^2\times \SSS^1)\right) \bigsqcup \bar B^2\times \SSS^1
\]
and form the quotient space
\[
\widehat M = \left( \left(M \setminus \phi(B^2\times \SSS^1)\right) \bigsqcup \bar B^2\times \SSS^1\right)\Big/{\sim_h},
\]
where the equivalence relation $\sim_h$ is the equivalence relation generated by the condition that $y\sim_h (x,e^{i \theta})$ if $y = h(x,e^{i\theta})$ for $y\in M$ and $(x,e^{i \theta}) \in (\partial B^2)\times \SSS^1$. In what follows, we use the common notation
\[
\widehat M =  \left(M \setminus \phi(B^2\times \SSS^1)\right) {\bigcup}_h \bar B^2\times \SSS^1
\]
for the space $\widehat M$. Please see Figure~\ref{fig:gluing2d} for a schematic of the gluing map $h$.

\begin{figure}[ht]
    \centering
    \includegraphics[scale=0.4]{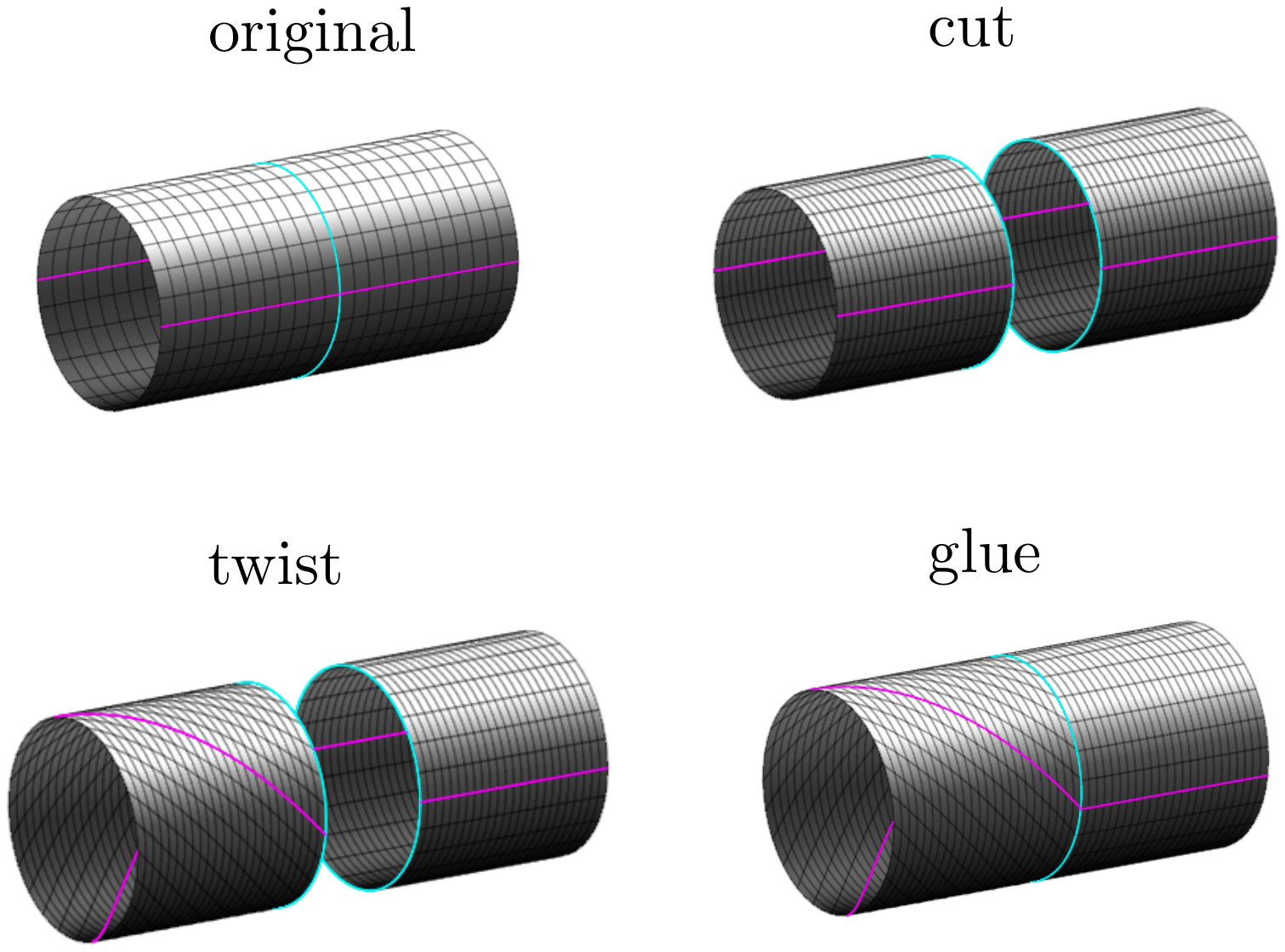}
    \caption{This figure schematically depicts the gluing homeomorphism $h$.}
    \label{fig:gluing2d}
\end{figure}

A surgery over a link $L$ in $M$ is obtained similarly as follows. The additional condition we impose is that the solid $3$-tori neighbourhoods of the circles in the link $L$ do not meet. For brevity, we introduce the following notation. Let $L = S_1\cup \cdots \cup S_J$ be a link with $J$ circles in $M$, and let $T_1,\ldots, T_J$ be mutually disjoint solid $3$-tori having circles $S_1,\ldots, S_J$ as their core curves, respectively. Let also $h \colon (\partial \bar B^2 \times \SSS^1)\times \{1,\ldots, J\} \to (\partial T_1\cup \cdots \cup \partial T_J)$ be a homeomorphism. Then 
\[
\widehat M = \left( \left( M\setminus \interior( T_1 \cup \cdots \cup T_J\right) \bigsqcup \bar B^2\times \SSS^1 \times \{1,\ldots, J\}\right)\Big/{\sim_h},
\]
is a \emph{space obtained from $M$ as a surgery over the link $L$}. Here the equivalence relation with respect to $h$ is analogous to the case of one circle, that is, the equivalence relation $\sim_h$ is the equivalence relation $\sim_h$ generated by the condition $y\sim_h (x,t,j)$, where $y=h(x,e^{i\theta},j)$ for $y\in M$, $(x,e^{i\theta})\in \bar B^2\times \SSS^1$ and $j=1,\ldots, J$. Again, we denote
\[
\widehat M = \left(M\setminus \interior( T_1 \cup \cdots \cup T_J)\right) {\bigcup}_h \bar B^2\times \SSS^1 \times \{1,\ldots, J\}.
\]
 
In what follows, let 
\begin{align*}
\pi_h \colon & \left( M\setminus \interior( T_1 \cup \cdots \cup T_J\right) \bigsqcup \bar B^2\times \SSS^1 \times \{1,\ldots, J\} \\
&\to ( M\setminus \interior( T_1 \cup \cdots \cup T_J)) {\bigcup}_h \bar B^2\times \SSS^1 \times \{1,\ldots, J\}
\end{align*}
be the canonical projection, i.e.\ the quotient map $x\mapsto [x]$ for the equivalence relation $\sim_h$.

For each $j=1,\ldots, J$, we also denote $\widehat T_j = \pi_h(\bar B^2\times \SSS^1\times \{j\})$ and $\widehat S_j = \pi_h(\{0\}\times \SSS^1\times \{j\})$. Note that each $\widehat T_j$ is a solid $3$-torus and each $\widehat S_j$ the core curve of $\widehat T_j$. In particular, 
\begin{equation}
\label{eq:widehat-L}
\widehat L = \widehat S_1\cup \cdots \cup \widehat S_J
\end{equation}
is a link in $\widehat M$.

Having this terminology at our disposal, we may now give a precise statement of the Lickorish--Wallace theorem. 

\begin{theorem}[Lickorish \cite{Lickorish}, Wallace \cite{Wallace}]
\label{thm:LW-precise}
Let $M$ be a closed, connected, and orientable $3$-manifold. Then there exists a link $L = S_1 \cup \cdots \cup S_J \subset \SSS^3$, mutually disjoint solid $3$-tori $T_1,\ldots, T_J$ in $\SSS^3$ with core curves $S_1,\ldots, S_J$, respectively, and a homeomorphism 
\[
h \colon (\partial \bar B^2\times \SSS^1) \times \{1,\ldots, J\} \to (\partial T_1 \cup \cdots \cup \partial T_J)
\]
for which 
\begin{equation}
\label{eq:LW-precise}
M \approx \widehat M = \left(\SSS^3\setminus \interior(T_1\cup \cdots \cup T_J)\right) {\bigcup}_h \bar B^2\times \SSS^1\times \{1,\ldots,J\}.
\end{equation}
\end{theorem}

\subsection{Embeddings satisfying \eqref{explicit diffeo}}

Although our primary aim is to obtain a smooth embedding $\widehat M\setminus \widehat L \to \RR^3$, in our applications we need the embedding to have controlled derivative close to the link $\widehat L$ as stated in \eqref{explicit diffeo}. For the definition, we use polar coordinates in the $\bar B^2$ factor of $\bar B^2 \times \SSS^1$. In practice, this means that we will use three coordinates $(r,\theta,s)$, $r \in [0,1]$, $\theta,s \in \RR$, to denote a point in $\bar B^2 \times \SSS^1$ instead of just two $(x,e^{i\psi})$, $x \in \bar B^2$, $e^{i\psi} \in \SSS^1$. As usual, these two coordinate systems are related to each other with the formulas $x = (r\cos{\theta},r\sin{\theta})$ and $\psi = s$.

Now we can give the precise definition for the controlled derivative.
\begin{defn}
Let $\widehat M$ and $M$ be closed $3$-manifolds and $\widehat L \subset \widehat M$ and $L \subset M$ links. We say that the derivatives of a diffeomorphism $F\colon \widehat M \setminus \widehat L \to M \setminus L$ are \textbf{controlled close to the link $\widehat L$}, if there exist mutually disjoint solid $3$-tori $\widehat T_1, \dots, \widehat T_J \subset \widehat M$ and $T_1, \dots, T_J \subset M$ with parametrizations $\widehat \phi: \bar B^2 \times \SSS^1 \times \{1,\dots,J\} \to \widehat T_1 \cup \dots \cup \widehat T_J$ and $\phi: \bar B^2 \times \SSS^1 \times \{1,\dots,J\} \to T_1 \cup \dots \cup T_J$, for which $\widehat L = \widehat \phi(\{0\} \times \SSS^1 \times \{1,\dots,J\})$, $L = \phi(\{0\} \times \SSS^1 \times \{1,\dots,J\})$, $F((\widehat T_1 \cup \dots \cup \widehat T_J) \setminus \widehat L) \subset T_1 \cup \dots \cup T_J$ and the partial derivatives of the composite map
\[
\pi = \phi^{-1} \circ F \circ \widehat \phi|_{(\bar B^2 \setminus \{0\}) \times \SSS^1 \times \{j\}}\colon (\bar B^2 \setminus \{0\}) \times \SSS^1 \to (\bar B^2 \setminus \{0\}) \times \SSS^1
\]
are bounded for each $j = 1,\dots,J$, i.e. there are positive constants $C_{rr}$, $C_{r\theta}$, $C_{rs}$, $C_{\theta r}, \dots$ which satisfy
\[
\left| \frac{\partial \pi_r}{\partial r} \right| \leq C_{rr},
\quad\left| \frac{\partial \pi_r}{\partial \theta} \right| \leq C_{r\theta},
\quad\left| \frac{\partial \pi_r}{\partial s} \right| \leq C_{rs},
\quad\left| \frac{\partial \pi_\theta}{\partial r} \right| \leq C_{\theta r},
\,\,\,\dots
\]
\end{defn}

To be able to construct a diffeomorphism with this property for our needs, we require also our gluing homeomorphisms to be diffeomorphisms. Our first preliminary result is the following lemma, which states that for every gluing homeomorphism there exists a gluing diffeomorphism that yields in the surgery the same $3$-manifold up to a homeomorphism.

\begin{lemma}
\label{lemma:gluing-diffeomorphism}
Let $M$ be a compact, smooth $3$-manifold and $\widehat M$ a $3$-manifold obtained from $M$ by a surgery along mutually disjoint solid tori $T_1,\dots,T_J \subset M$ with gluing homeomophism $h\colon (\partial \bar B^2 \times \SSS^1)\times \{1,\ldots, J\} \to (\partial T_1\cup \cdots \cup \partial T_J)$. Then there exists another gluing homeomorphism $h'$, which is a diffeomorphism and for which the manifold
\[
\widehat M' = (M \setminus \interior (T_1 \cup \cdots T_J)) {\bigcup}_{h'} \bar B^2\times \SSS^1 \times \{1,\ldots, J\}
\]
is homeomorphic to $\widehat M$. Furthermore, $\widehat M'$ has a smooth structure induced by the smooth structures of $M\setminus\interior (T_1 \cup \cdots \cup T_J)$ and $\bar B^2\times \SSS^1 \times \{1,\ldots, J\}$.
\end{lemma}

The proof is based on a classical two dimensional smoothing result that all homeomorphisms between surfaces are isotopic to diffeormorphisms; see Baer \cite{Baer}, Epstein \cite{Epstein}, or an unpublished short proof due to Hatcher \cite{Hatcher-survey}. We formulate the needed result as follows:

\begin{theorem}
\label{thm:isotopy}
Let $M$ and $M'$ be closed, smooth $2$-manifolds, and $f\colon M \to M'$ a homeomorphism. Then there exists a diffeomorphism $f'\colon M \to M'$ that is isotopic to $f$, i.e. there exists a map
\[
F\colon M \times [0,1] \to M',
\]
for which $F(x,0) = f(x)$ and $F(x,1) = f'(x)$ for all $x \in M$ and the map $F_t\colon M \to M'$, $x \mapsto F(x,t)$ is an embedding for each $t \in [0,1]$.
\end{theorem}

\begin{proof}[Proof of Lemma~\ref{lemma:gluing-diffeomorphism}]
Let $h'\colon (\partial \bar B^2 \times \SSS^1)\times \{1,\ldots, J\} \to (\partial T_1\cup \cdots \cup \partial T_J)$ be a diffeomorphism isotopic to $h$ as in Theorem~\ref{thm:isotopy}. It is a well-known fact that if the gluing map is a diffeomorphism, the smooth structures of the manifolds glued together induce a unique smooth structure on the resulting manifold (see for example Hirsch \cite[Section 8.2]{Hirsch}).

Now it remains to show that $\widehat M'$ is homeomorphic to $\widehat M$. We will define the homeomorphism $\omega\colon \widehat M' \to \widehat M$ piecewise. Write $\bar B^2(\frac{1}{2}) = \{(x,y) \in \RR^2 : \sqrt{x^2+y^2} \leq \frac{1}{2} \}$. Let $\omega|_{M\setminus\interior (T_1 \cup \cdots \cup T_J)}$ and $\omega|_{\bar B^2(\frac{1}{2})\times \SSS^1 \times \{1,\ldots, J\}}$ be the identity maps. Let $F\colon (\partial \bar B^2 \times \SSS^1)\times \{1,\ldots, J\} \times [0,1] \to (\partial T_1\cup \cdots \cup \partial T_J)$ be the isotopy between $h$ and $h'$. Define $\omega|_{(\bar B^2 \setminus B^2(\frac{1}{2}))\times \SSS^1 \times \{1,\ldots, J\}}$ to be
\[
(x,e^{i\theta},j) \mapsto |x| \cdot (h^{-1} \circ F)(x/|x|,e^{i\theta},j,2(|x|-\tfrac{1}{2})),
\]
where we use the notation $a \cdot (x,e^{i\theta},j) = (ax,e^{i\theta},j)$ with $a \in [0,1]$ and $(x,e^{i\theta},j) \in \partial \bar B^2 \times \SSS^1 \times \{1,\dots,J\}$.

Now fix $t \in [0,1]$. The map $F_t$ is an embedding and hence injective. This implies with the domain invariance theorem that $F_t$ is an open map. Now fix $j \in \{1,\dots,J\}$. The map $F$ is continuous and the set $\partial \bar B^2 \times \SSS^1 \times \{j\} \times [0,1]$ is connected, so its image $F(\partial \bar B^2 \times \SSS^1 \times \{j\} \times [0,1])$ is also connected and hence contained in some $\partial T_{j'}$. We may assume that $j'=j$. The set $\partial \bar B^2 \times \SSS^1 \times \{j\}$ is open and compact and hence also its image $F_t(\partial \bar B^2 \times \SSS^1 \times \{j\})$ is open and compact and thus open and closed. This implies that it covers the whole $\partial T_j$. Hence $F_t$ is surjective.

We have shown that $F_t$ is bijective for each $t \in [0,1]$. This together with the initial assumption that $h$ is a homeomorphism and $h^{-1}$ hence bijective implies that $\omega|_{(\bar B^2 \setminus B^2(\frac{1}{2}))\times \SSS^1 \times \{1,\ldots, J\}}$ is bijective.

We will show next that $\omega$ is well-defined. Let $(x,e^{i\theta},j) \in \partial \bar B^2 \times \SSS^1 \times \{1,\ldots, J\}$. Now we have
\begin{align*}
\omega|_{(\bar B^2 \setminus B^2(\frac{1}{2}))\times \SSS^1 \times \{1,\ldots, J\}}(x,e^{i\theta},j)
&= |x| \cdot (h^{-1} \circ F)(x/|x|,e^{i\theta},j,2(|x|-\tfrac{1}{2}))\\
&= 1 \cdot (h^{-1} \circ F)(x/1,e^{i\theta},j,2(1-\tfrac{1}{2}))\\
&= (h^{-1} \circ F)(x,e^{i\theta},j,1).
\end{align*}
Thus
\begin{align*}
h(\omega|_{(\bar B^2 \setminus B^2(\frac{1}{2}))\times \SSS^1 \times \{1,\ldots, J\}}(x,e^{i\theta},j))
&= F(x,e^{i\theta},j,1) \\
& = h'(x,e^{i\theta},j) \\
&= \omega|_{M\setminus\interior (T_1 \cup \cdots \cup T_J)}(h'(x,e^{i\theta},j)).
\end{align*}
Hence the definitions of $\omega|_{(\bar B^2 \setminus B^2(\frac{1}{2}))\times \SSS^1 \times \{1,\ldots, J\}}$ and $\omega|_{M\setminus\interior (T_1 \cup \cdots \cup T_J)}$ agree on the intersection $((\bar B^2 \setminus B^2(\frac{1}{2}))\times \SSS^1 \times \{1,\ldots, J\}) \bigcap (M\setminus\interior (T_1 \cup \cdots \cup T_J))$.

Let now $(x,e^{i\theta},j) \in (\bar B^2(\frac{1}{2}) \setminus B^2(\frac{1}{2})) \times \SSS^1 \times \{1,\dots,J\}$. We have
\begin{align*}
&\omega|_{(\bar B^2 \setminus B^2(\frac{1}{2}))\times \SSS^1 \times \{1,\ldots, J\}}(x,e^{i\theta},j)
= \tfrac{1}{2} \cdot (h^{-1} \circ F)(x/\tfrac{1}{2},e^{i\theta},j,2(\tfrac{1}{2}-\tfrac{1}{2})) \\
&= \tfrac{1}{2} \cdot (h^{-1} \circ F)(2x,e^{i\theta},j,0)
= \tfrac{1}{2} \cdot (h^{-1} \circ h)(2x,e^{i\theta},j) \\
&= \tfrac{1}{2} \cdot (2x,e^{i\theta},j) = (x,e^{i\theta},j)
= \omega|_{\bar B^2(\frac{1}{2})\times \SSS^1 \times \{1,\ldots, J\}}(x,e^{i\theta},j).
\end{align*}
Hence the definitions of $\omega|_{(\bar B^2 \setminus B^2(\frac{1}{2}))\times \SSS^1 \times \{1,\ldots, J\}}$ and $\omega|_{\bar B^2(\frac{1}{2})\times \SSS^1 \times \{1,\ldots, J\}}$ also agree on the intersection $((\bar B^2 \setminus B^2(\frac{1}{2}))\times \SSS^1 \times \{1,\ldots, J\}) \bigcap (\bar B^2(\frac{1}{2})\times \SSS^1 \times \{1,\ldots, J\})$.

All of its pieces are continuous, so also the whole map $\omega$ is continuous. Furthermore, the manifold $\widehat M'$ is compact, so we conclude that $\omega$ is a homeomorphism.
\end{proof}

Having now all the terminology and this lemma at our disposal, we may formulate an embedding lemma: 

\begin{lemma}
\label{lemma:3-dim-embedding}
Let $M$ be a closed, smooth $3$-manifold and $\widehat M$ a $3$-manifold obtained from $M$ by a surgery in link $L\subset M$. Then there exists a link $\widehat L$ in $\widehat M$ and a diffeomorphism $\widehat M\setminus \widehat L \to M \setminus L$. Furthermore, the derivatives of this diffeomorphism are controlled close to the link $\widehat L$.
\end{lemma}

\begin{proof}

For the argument, we assume that link $L$ has $J$ circles $S_1,\ldots,S_J$ and these circles are core curves of mutually disjoint solid $3$-tori $T_1,\ldots, T_J$ in $M$. Let $h \colon (\partial \bar B^2 \times \SSS^1)\times \{1,\ldots, J\} \to (\partial T_1\cup \cdots \cup \partial T_J)$ be the gluing homeomorphism for which 
\[
\widehat M = (M \setminus \interior (T_1 \cup \cdots T_J)) {\bigcup}_h \bar B^2\times \SSS^1 \times \{1,\ldots, J\}
\]
and let $\widehat T_1,\ldots, \widehat T_J$ be solid $3$-tori and $\widehat S_1,\ldots, \widehat S_J$ their core curves, respectively, as above.

We follow the idea of the proof of Proposition \ref{prop:2-dim-embedding} and define the diffeomorphism $F\colon \widehat M \setminus \widehat L \to M\setminus L$ in parts. On $M\cap \widehat M$, we set $F$ to be the identity. Thus it suffices to define $F$ on $\widehat T_j \setminus \widehat S_j$ for each $j = 1,\ldots, J$. 
To simplify the notation, let $\Omega_j = T_j \setminus S_j$ and $\widehat \Omega_j = \widehat T_j \setminus \widehat S_j$ for each $j=1,\ldots, J$. Let also $\Omega = \Omega_1 \cup \cdots \cup \Omega_J$ and $\widehat \Omega = \widehat \Omega_1 \cup \cdots \cup \widehat \Omega_J$.

Let $\phi  \colon \bar B^2\times \SSS^1 \times \{1,\ldots, J\}\to T_1\cup \cdots \cup T_J$ be a diffeomorphism simultaneously parametrizing each solid torus $T_1,\ldots, T_J$. By Lemma~\ref{lemma:gluing-diffeomorphism}, we may assume the gluing homeomorphism $h$ to be a diffeomorphism. Then it induces a diffeomorphism
\[
h' \colon (\partial \bar B^2)\times \SSS^1 \times \{1,\ldots, J\} \to  (\partial \bar B^2)\times \SSS^1 \times \{1,\ldots, J\}
\]
satisfying
\[
\phi \circ h' = h.
\]
Let now 
\[
H \colon (\bar B^2\setminus \{0\}) \times \SSS^1\times \{1,\ldots, J\} \to (\bar B^2\setminus \{0\}) \times \SSS^1\times \{1,\ldots, J\}
\]
be the homeomorphism 
\[
(x,e^{i \theta},j) \mapsto |x| \cdot h'(x/|x|, e^{i\theta}, j)
\]
extending the homeomorphism $h'$, where we are again using the notation $a \cdot (x,e^{i\theta},j) = (ax,e^{i\theta},j)$.

Finally, let 
\[
\psi \colon \bar B^2 \times \SSS^1\times \{1,\ldots, J\} \to \widehat T_1 \cup \cdots \cup \widehat T_J
\]
be the restriction of the quotient map $\pi_h$ and define the restriction
\[
F|_{\widehat \Omega} \colon \widehat \Omega \to \Omega
\]
by the formula
\[
F|_{\widehat \Omega} = \phi \circ H \circ \psi^{-1}|_{\widehat \Omega}.
\]
Since $\phi \circ H \circ \psi^{-1}|_{\partial \widehat \Omega}$ is the identity, the map $F$ is well-defined. Since restrictions $F|_{\widehat M \setminus \interior \widehat \Omega}$ and $F|_{\widehat \Omega}$ are continuous, we have that $F$ is continuous. Clearly, $F$ is also bijective. Since $F$ extends to a continuous bijection over the link $\widehat L$ to a continuous bijection $\widehat M \to M$ and $\widehat M$ is compact, we conclude that $F$ is a homeomorphism.


Recall that we defined the map $F$ so that the restriction $F|_{\widehat M \setminus \interior \widehat \Omega}$ is the identity. Now the domain and codomain of $F|_{\widehat M \setminus \interior \widehat \Omega}$ have the same differential structure, so this restriction is furthermore a diffeomorphism. For the restriction $F|_{\widehat \Omega}$, we have $F|_{\widehat \Omega} = \phi \circ H \circ \psi^{-1}|_{\widehat \Omega}$. Since the map $h'$ is a diffeomorphism, also $H$ is a diffeomorphism. In addition, the maps $\phi$ and $\psi$ are smooth embeddings, so as their composite map, the restriction $F|_{\widehat \Omega}$ is a diffeomorphism.

We have now shown that the both restrictions
\[
F|_{\widehat M \setminus \interior \widehat \Omega}\colon \widehat M \setminus \interior \widehat \Omega \to M \setminus \interior \Omega \qquad\text{and }\quad F|_{\widehat \Omega}\colon \widehat \Omega \to \Omega
\]
are diffeomorphisms. It then follows that there exists a diffeomorphism $\widehat M\setminus \widehat L \to M \setminus L$ that agrees with $F$ on the subset $\widehat \Omega$ (see for example Hirsch \cite[Theorem 8.1.9]{Hirsch}). For the rest of this proof, we will use $F$ to denote this diffeomorphism.

It remains to show that the derivatives of the diffeomorphism $F$ are controlled close to the link $\widehat L$. We have
\begin{align*}
\pi &= \phi^{-1} \circ F \circ \psi|_{(\bar B^2 \setminus \{0\}) \times \SSS^1 \times \{j\}}\\
    &= \phi^{-1} \circ \phi \circ H \circ \psi^{-1} \circ \psi|_{(\bar B^2 \setminus \{0\}) \times \SSS^1 \times \{j\}}\\
    &= H|_{(\bar B^2 \setminus \{0\}) \times \SSS^1 \times \{j\}}.
\end{align*}
We want to now use polar coordinates for $\bar B^2$, so we use the notation
\[
h'(\theta,s,j) = (h'_\theta(\theta,s,j),h'_s(\theta,s,j),h'_j(\theta,s,j)),
\]
where $h'_\theta$, $h'_s$ and $h'_j$ are the coordinate functions of $h'$. Now we have for the map $H$ the formula
\[
H(r,\theta,s,j) = (r,h'_\theta(\theta,s,j),h'_s(\theta,s,j),j).
\]
Then we may calculate the partial derivatives. For the derivative $\frac{\partial \pi_r}{\partial r}$, we obtain
\[
\frac{\partial \pi_r}{\partial r} (r,\theta,s) = \frac{\partial H_r}{\partial r} (r,\theta,s)
= \frac{\partial r}{\partial r} = 1.
\]
For the derivative $\frac{\partial \pi_\theta}{\partial s}$, we obtain
\[
\frac{\partial \pi_\theta}{\partial s} (r,\theta,s)
= \frac{\partial H_\theta}{\partial s} (r,\theta,s,j)
= \frac{\partial h'_\theta}{\partial s} (r,\theta,s,j).
\]
The derivative $\frac{\partial h'_\theta}{\partial s}$ is continuous and its domain $\bar B^2 \times \SSS^1 \times \{1,\dots,J\}$ is compact, so also its image $\frac{\partial h'_\theta}{\partial s} (\bar B^2 \times \SSS^1 \times \{1,\dots,J\})$ is compact and hence bounded. The remaining partial derivatives can be shown to be bounded in a similar manner.
\end{proof}

\subsection{Proof of Proposition \ref{smooth-embedding}}

The controlled version of Proposition \ref{smooth-embedding} reads as follows:

\begin{theorem}
\label{thm:smooth-embedding-precise}
Let $M$ be a closed, connected, and orientable $3$-manifold. Then there exist links $L \subset M$ and $\widetilde L \subset \SSS^3$, and a diffeomorphism $M\setminus L \to \SSS^3\setminus \widetilde L$. Furthermore, the derivatives of this diffeomorphism are controlled close to the link $L$.
\end{theorem}

The last ingredient of the proof of Theorem \ref{thm:smooth-embedding-precise} is the smoothing theorem due to Munkres, which states that homeomorphic smooth $3$-manifolds are diffeomorphic.

\begin{theorem} 
Let $M$ and $N$ be smooth $3$-manifolds. If there exists a homeomorphism $M\to N$, then there exists a diffeomorphism $M\to N$.
\end{theorem}

\begin{proof}[Proof of Theorem \ref{thm:smooth-embedding-precise}]
Let $\widetilde L$ be the link in $\SSS^3$ and $\widehat M$ the closed, connected, and orientable $3$-manifold given by Theorem \ref{thm:LW-precise}. Let also $\widehat L$ be the link in $\widehat M$ induced by the surgery of $\SSS^3$ along the link $\widetilde L$. 
By Lemma~\ref{lemma:gluing-diffeomorphism}, we may assume $\widehat M$ to be a smooth manifold. Then, since $M$ and $\widehat M$ are homeomorphic smooth $3$-manifolds, there exists a diffeomorphism $\Phi \colon M \to \widehat M$. Let $L = \Phi^{-1} \widehat L$. 
Finally, there exists a diffeomorphism $\psi\colon \widehat M\setminus \widehat L \to \SSS^3\setminus \widetilde L$ by Lemma \ref{lemma:3-dim-embedding}. Let now $F\colon M\setminus L \to \SSS^3\setminus \widetilde L$ be the diffeomorphism $F = \psi \circ \Phi$.

Let $\widetilde T$ be the union of the solid tori in $\SSS^3$ that the surgery is performed along and let $\widetilde \rho\colon \bar B^2 \times \SSS^1 \times \{1,\dots,J\} \to \widetilde T$ be its parametrization. Let $\widehat T \subset \widehat M$ be the union of the solid tori that were attached to $\SSS^3 \setminus \interior \widetilde T$ in the surgery and let $\phi\colon \bar B^2 \times \SSS^1 \times \{1,\dots,J\} \to \widehat T$ be its parametrization induced by the quotient map $\pi_h$.
Let $\rho = \Phi^{-1} \circ \phi$. We now have
\begin{align*}
\pi &= \widetilde \rho^{-1} \circ F \circ \rho|_{(\bar B^2 \setminus \{0\}) \times \SSS^1 \times \{j\}}
    = \widetilde \rho^{-1} \circ \psi \circ \Phi \circ \rho|_{(\bar B^2 \setminus \{0\}) \times \SSS^1 \times \{j\}} \\
    &= \widetilde \rho^{-1} \circ \psi \circ \phi \circ \phi^{-1} \circ \Phi \circ \rho|_{(\bar B^2 \setminus \{0\}) \times \SSS^1 \times \{j\}}.
\end{align*}
Using the previous formula and the chain rule for derivatives, we have for the derivative $\frac{\partial \pi_r}{\partial \theta}$
\begin{align*}
\frac{\partial \pi_r}{\partial \theta}(r,\theta,s)
&= \frac{\partial (\widetilde \rho^{-1} \circ \psi \circ \phi \circ \phi^{-1} \circ \Phi \circ \rho)_r}{\partial \theta}(r,\theta,s) \\
&= \frac{\partial (\phi^{-1} \circ \Phi \circ \rho)_r}{\partial \theta}(r,\theta,s)
\cdot \frac{\partial (\widetilde \rho^{-1} \circ \psi \circ \phi)_r}{\partial r}(\phi^{-1} \circ \Phi \circ \rho(r,\theta,s)) \\
&\quad+ \frac{\partial (\phi^{-1} \circ \Phi \circ \rho)_\theta}{\partial \theta}(r,\theta,s)
\cdot \frac{\partial (\widetilde \rho^{-1} \circ \psi \circ \phi)_r}{\partial \theta}(\phi^{-1} \circ \Phi \circ \rho(r,\theta,s)) \\
&\quad+ \frac{\partial (\phi^{-1} \circ \Phi \circ \rho)_s}{\partial \theta}(r,\theta,s)
\cdot \frac{\partial (\widetilde \rho^{-1} \circ \psi \circ \phi)_r}{\partial s}(\phi^{-1} \circ \Phi \circ \rho(r,\theta,s)).
\end{align*}
We know by Lemma~\ref{lemma:3-dim-embedding} that the derivatives of the diffeomorphism $\psi$ are controlled close to the link $\widehat L$. Hence the partial derivatives of the composite map $\widetilde \rho^{-1} \circ \psi \circ \phi$ are bounded. The partial derivatives of the composite map $\phi^{-1} \circ \Phi \circ \rho$, on the other hand, are continuous and their domain $\bar B^2 \times \SSS^1 \times \{1,\dots,J\}$ is compact, so also their images are compact and hence bounded. Since all these partial derivatives are bounded, it follows that also $\frac{\partial \pi_r}{\partial \theta}$ is bounded. The remaining partial derivatives of the map $\pi$ can be shown to be bounded using similar argument.
\end{proof}





\bibliographystyle{abbrv}
\bibliography{universe-references}

\begin{thebibliography}{10}

\bibitem{AluEngheta2005}
A.~Al\`u and N.~Engheta.
\newblock Achieving transparency with plasmonic and metamaterial coatings.
\newblock {\em Phys. Rev. E}, 72:016623, Jul 2005.

\bibitem{AE2}
A.~Al\`u and N.~Engheta.
\newblock Cloaking a sensor.
\newblock {\em Phys. Rev. Lett.}, 102:233901, Jun 2009.

\bibitem{Ammari-etal2013}
H.~Ammari, G.~Ciraolo, H.~Kang, H.~Lee, and G.~W. Milton.
\newblock Spectral theory of a {N}eumann-{P}oincar\'{e}-type operator and
  analysis of cloaking due to anomalous localized resonance.
\newblock {\em Arch. Ration. Mech. Anal.}, 208(2):667--692, 2013.

\bibitem{Ammari2016}
H.~Ammari, Y.~Deng, and P.~Millien.
\newblock Surface plasmon resonance of nanoparticles and applications in
  imaging.
\newblock {\em Arch. Ration. Mech. Anal.}, 220(1):109--153, 2016.

\bibitem{AmmariPart2}
H.~Ammari, H.~Kang, H.~Lee, and M.~Lim.
\newblock Enhancement of near-cloaking. {P}art {II}: {T}he {H}elmholtz
  equation.
\newblock {\em Comm. Math. Phys.}, 317(2):485--502, 2013.

\bibitem{AmmariPart1}
H.~Ammari, H.~Kang, H.~Lee, and M.~Lim.
\newblock Enhancement of near cloaking using generalized polarization tensors
  vanishing structures. {P}art {I}: {T}he conductivity problem.
\newblock {\em Comm. Math. Phys.}, 317(1):253--266, 2013.

\bibitem{AKL-cosmic}
S.~Antoniou, L.~H. Kauffman, and S.~Lambropoulou.
\newblock Topological surgery in cosmic phenomena.
\newblock {\em Adv. Theor. Math. Phys.}, 23(3):701--765, 2019.

\bibitem{AKL-BH}
S.~Antoniou, L.~H. Kauffman, and S.~Lambropoulou.
\newblock Black holes and topological surgery.
\newblock {\em J. Knot Theory Ramifications}, 29(10):2042010, 6, 2020.

\bibitem{Astala}
K.~Astala, M.~Lassas, and L.~P\"{a}iv\"{a}rinta.
\newblock The borderlines of invisibility and visibility in {C}alder\'{o}n's
  inverse problem.
\newblock {\em Anal. PDE}, 9(1):43--98, 2016.

\bibitem{Baer}
R.~Baer.
\newblock Isotopie von {K}urven auf orientierbaren, geschlossenen {F}l\"{a}chen
  und ihr {Z}usammenhang mit der topologischen {D}eformation der {F}l\"{a}chen.
\newblock {\em J. Reine Angew. Math.}, 159:101--116, 1928.

\bibitem{BaoLiu2014}
G.~Bao and H.~Liu.
\newblock Nearly cloaking the electromagnetic fields.
\newblock {\em SIAM J. Appl. Math.}, 74(3):724--742, 2014.

\bibitem{BaoLiuZou}
G.~Bao, H.~Liu, and J.~Zou.
\newblock Nearly cloaking the full {M}axwell equations: cloaking active
  contents with general conducting layers.
\newblock {\em J. Math. Pures Appl. (9)}, 101(5):716--733, 2014.

\bibitem{braides-handbook}
A.~Braides.
\newblock A handbook of $\gamma$-convergence.
\newblock In {\em Handbook of Differential Equations: stationary partial
  differential equations}, volume~3, pages 101--213. Elsevier, 2006.

\bibitem{ChenChan2007}
H.~Chen and C.~T. Chan.
\newblock Acoustic cloaking in three dimensions using acoustic metamaterials.
\newblock {\em Applied Physics Letters}, 91(18):183518, 2007.

\bibitem{ChenRotate}
H.~Chen and C.~T. Chan.
\newblock Transformation media that rotate electromagnetic fields.
\newblock {\em Applied Physics Letters}, 90(24):241105, 2007.

\bibitem{DLU2017-2}
Y.~Deng, H.~Liu, and G.~Uhlmann.
\newblock Full and partial cloaking in electromagnetic scattering.
\newblock {\em Arch. Ration. Mech. Anal.}, 223(1):265--299, 2017.

\bibitem{DLU2017}
Y.~Deng, H.~Liu, and G.~Uhlmann.
\newblock On regularized full- and partial-cloaks in acoustic scattering.
\newblock {\em Comm. Partial Differential Equations}, 42(6):821--851, 2017.

\bibitem{PhysClosed}
E.~Di~Valentino, A.~Melchiorri, and J.~Silk.
\newblock Planck evidence for a closed universe and a possible crisis for
  cosmology.
\newblock {\em Nature Astronomy}, 4(2):196--203, 2020.

\bibitem{edmunds-evans}
D.~Edmunds and W.~Evans.
\newblock {\em Spectral Theory and Differential Operators}.
\newblock Oxford mathematical monographs. Oxford University Press, 2018.

\bibitem{Epstein}
D.~B.~A. Epstein.
\newblock Curves on {$2$}-manifolds and isotopies.
\newblock {\em Acta Math.}, 115:83--107, 1966.

\bibitem{Evans-PDE}
L.~C. Evans.
\newblock {\em Partial differential equations}, volume~19 of {\em Graduate
  Studies in Mathematics}.
\newblock American Mathematical Society, Providence, RI, second edition, 2010.

\bibitem{FKR2014}
D.~Faraco, Y.~Kurylev, and A.~Ruiz.
\newblock {$G$}-convergence, {D}irichlet to {N}eumann maps and invisibility.
\newblock {\em J. Funct. Anal.}, 267(7):2478--2506, 2014.

\bibitem{GilbargTrudinger}
D.~Gilbarg and N.~S. Trudinger.
\newblock {\em Elliptic partial differential equations of second order}.
\newblock Classics in Mathematics. Springer-Verlag, Berlin, 2001.
\newblock Reprint of the 1998 edition.

\bibitem{GKLU-superdim}
A.~Greenleaf, H.~Kettunen, Y.~Kurylev, M.~Lassas, and G.~Uhlmann.
\newblock Superdimensional metamaterial resonators from sub-{R}iemannian
  geometry.
\newblock {\em SIAM J. Appl. Math.}, 78(1):437--456, 2018.

\bibitem{GKLU-Phys}
A.~Greenleaf, Y.~Kurylev, M.~Lassas, and G.~Uhlmann.
\newblock Electromagnetic wormholes and virtual magnetic monopoles from
  metamaterials.
\newblock {\em Phys. Rev. Lett.}, 99:183901, Oct 2007.

\bibitem{GKLU-fullwave}
A.~Greenleaf, Y.~Kurylev, M.~Lassas, and G.~Uhlmann.
\newblock Full-wave invisibility of active devices at all frequencies.
\newblock {\em Comm. Math. Phys.}, 275(3):749--789, 2007.

\bibitem{GKLU-wormhole}
A.~Greenleaf, Y.~Kurylev, M.~Lassas, and G.~Uhlmann.
\newblock Electromagnetic wormholes via handlebody constructions.
\newblock {\em Comm. Math. Phys.}, 281(2):369--385, 2008.

\bibitem{GKLU-survey}
A.~Greenleaf, Y.~Kurylev, M.~Lassas, and G.~Uhlmann.
\newblock Cloaking devices, electromagnetic wormholes, and transformation
  optics.
\newblock {\em SIAM Rev.}, 51(1):3--33, 2009.

\bibitem{GKLU-invisibility}
A.~Greenleaf, Y.~Kurylev, M.~Lassas, and G.~Uhlmann.
\newblock Invisibility and inverse problems.
\newblock {\em Bull. Amer. Math. Soc. (N.S.)}, 46(1):55--97, 2009.

\bibitem{GKLU-quantum}
A.~Greenleaf, Y.~Kurylev, M.~Lassas, and G.~Uhlmann.
\newblock Approximate quantum and acoustic cloaking.
\newblock {\em J. Spectr. Theory}, 1(1):27--80, 2011.

\bibitem{GKLU-physic}
A.~Greenleaf, Y.~Kurylev, M.~Lassas, and G.~Uhlmann.
\newblock Cloaking a sensor via transformation optics.
\newblock {\em Phys. Rev. E (3)}, 83(1):016603, 6, 2011.

\bibitem{GKLU7}
A.~Greenleaf, Y.~Kurylev, M.~Lassas, and G.~Uhlmann.
\newblock {Schrodinger's Hat: Electromagnetic, acoustic and quantum amplifiers
  via transformation optics}.
\newblock {\em Proc. Nat. Acad. Sci.}, 109:0169, 2012.

\bibitem{GLU-nonunique}
A.~Greenleaf, M.~Lassas, and G.~Uhlmann.
\newblock On nonuniqueness for {C}alder\'{o}n's inverse problem.
\newblock {\em Math. Res. Lett.}, 10(5-6):685--693, 2003.

\bibitem{Hatcher-survey}
A.~Hatcher.
\newblock The kirby torus trick for surfaces.
\newblock {\em arXiv:1312.3518 [math.GT]}, 2013.

\bibitem{Hatcher}
A.~Hatcher, C.~U. Press, and C.~U.~D. of~Mathematics.
\newblock {\em Algebraic Topology}.
\newblock Algebraic Topology. Cambridge University Press, 2002.

\bibitem{KKM-sobolev-book}
J.~Heinonen, T.~Kilpel\"{a}inen, and O.~Martio.
\newblock {\em Nonlinear potential theory of degenerate elliptic equations}.
\newblock Dover Publications, Inc., Mineola, NY, 2006.
\newblock Unabridged republication of the 1993 original.

\bibitem{Hirsch}
M.~W. Hirsch.
\newblock {\em Differential topology}, volume~33.
\newblock Springer Science \& Business Media, 2012.

\bibitem{Jost}
J.~Jost.
\newblock {\em Riemannian Geometry and Geometric Analysis}.
\newblock Hochschultext / Universitext. Springer, 1998.

\bibitem{kato-PLO}
T.~Kato.
\newblock {\em Perturbation theory for linear operators}, volume 132.
\newblock Springer Science \& Business Media, 2013.

\bibitem{KKM-sobolev0}
T.~Kilpel\"ainen, J.~Kinnunen, and O.~Martio.
\newblock Sobolev spaces with zero boundary values on metric spaces.
\newblock {\em Potential Anal.}, 12(3):233--247, 2000.

\bibitem{KOVW2010}
R.~V. Kohn, D.~Onofrei, M.~S. Vogelius, and M.~I. Weinstein.
\newblock Cloaking via change of variables for the {H}elmholtz equation.
\newblock {\em Comm. Pure Appl. Math.}, 63(8):973--1016, 2010.

\bibitem{KOVW-IP}
R.~V. Kohn, H.~Shen, M.~S. Vogelius, and M.~I. Weinstein.
\newblock Cloaking via change of variables in electric impedance tomography.
\newblock {\em Inverse Problems}, 24(1):015016, 21, 2008.

\bibitem{lachieze1995cosmic}
M.~Lachieze-Rey and J.-P. Luminet.
\newblock Cosmic topology.
\newblock {\em Physics Reports}, 254(3):135--214, 1995.

\bibitem{Lai2009}
Y.~Lai, J.~Ng, H.~Chen, D.~Han, J.~Xiao, Z.-Q. Zhang, and C.~T. Chan.
\newblock Illusion optics: The optical transformation of an object into another
  object.
\newblock {\em Phys. Rev. Lett.}, 102:253902, Jun 2009.

\bibitem{Landy2008}
N.~I. Landy, S.~Sajuyigbe, J.~J. Mock, D.~R. Smith, and W.~J. Padilla.
\newblock Perfect metamaterial absorber.
\newblock {\em Phys. Rev. Lett.}, 100:207402, May 2008.

\bibitem{LST2015}
M.~Lassas, M.~Salo, and L.~Tzou.
\newblock Inverse problems and invisibility cloaking for {FEM} models and
  resistor networks.
\newblock {\em Math. Models Methods Appl. Sci.}, 25(2):309--342, 2015.

\bibitem{LZ2011}
M.~Lassas and T.~Zhou.
\newblock Two dimensional invisibility cloaking for {H}elmholtz equation and
  non-local boundary conditions.
\newblock {\em Math. Res. Lett.}, 18(3):473--488, 2011.

\bibitem{LZ2016}
M.~Lassas and T.~Zhou.
\newblock The blow-up of electromagnetic fields in 3-dimensional invisibility
  cloaking for maxwell's equations.
\newblock {\em SIAM Journal on Applied Mathematics}, 76(2):457--478, 2016.

\bibitem{LLRU2015}
J.~Li, H.~Liu, L.~Rondi, and G.~Uhlmann.
\newblock Regularized transformation-optics cloaking for the {H}elmholtz
  equation: from partial cloak to full cloak.
\newblock {\em Comm. Math. Phys.}, 335(2):671--712, 2015.

\bibitem{LLLW2017}
X.~Li, J.~Li, H.~Liu, and Y.~Wang.
\newblock Electromagnetic interior transmission eigenvalue problem for
  inhomogeneous media containing obstacles and its applications to near
  cloaking.
\newblock {\em IMA J. Appl. Math.}, 82(5):1013--1042, 2017.

\bibitem{Lickorish}
W.~R. Lickorish.
\newblock A representation of orientable combinatorial 3-manifolds.
\newblock {\em Annals of Mathematics}, pages 531--540, 1962.

\bibitem{PhysOpen}
A.~R. Liddle and M.~Cort\^es.
\newblock Cosmic microwave background anomalies in an open universe.
\newblock {\em Phys. Rev. Lett.}, 111:111302, Sep 2013.

\bibitem{LiuSun2013}
H.~Liu and H.~Sun.
\newblock Enhanced near-cloak by {FSH} lining.
\newblock {\em J. Math. Pures Appl. (9)}, 99(1):17--42, 2013.

\bibitem{LiuUhlmann2015}
H.~Liu and G.~Uhlmann.
\newblock Regularized transformation-optics cloaking in acoustic and
  electromagnetic scattering.
\newblock In {\em Inverse problems and imaging}, volume~44 of {\em Panor.
  Synth\`eses}, pages 111--136. Soc. Math. France, Paris, 2015.

\bibitem{LiuZhou}
H.~Liu and T.~Zhou.
\newblock On approximate electromagnetic cloaking by transformation media.
\newblock {\em SIAM J. Appl. Math.}, 71(1):218--241, 2011.

\bibitem{LT2011-2}
H.~Liu and T.~Zhou.
\newblock Two dimensional invisibility cloaking via transformation optics.
\newblock {\em Discrete Contin. Dyn. Syst.}, 31(2):525--543, 2011.

\bibitem{luminet2008shape}
J.-P. Luminet.
\newblock The shape and topology of the universe.
\newblock {\em arXiv preprint arXiv:0802.2236}, 2008.

\bibitem{luminet2016survey}
J.-P. Luminet.
\newblock The status of cosmic topology after planck data.
\newblock {\em Universe}, 2(1):1, 2016.

\bibitem{luminet-finite}
J.-P. Luminet, J.~R. Weeks, A.~Riazuelo, R.~Lehoucq, and J.-P. Uzan.
\newblock Dodecahedral space topology as an explanation for weak wide-angle
  temperature correlations in the cosmic microwave background.
\newblock {\em Nature}, 425(6958):593--595, 2003.

\bibitem{dalmaso}
G.~Maso.
\newblock {\em An Introduction to $\Gamma$-Convergence}.
\newblock Progress in Nonlinear Differential Equations and Their Applications.
  Birkh{\"a}user Boston, 2012.

\bibitem{MiltonMcPhedran}
G.~W. Milton and R.~C. McPhedran.
\newblock Anomalous localized resonance and associated cloaking.
\newblock {\em SIAM News}, 51(6):6, 8, 2018.

\bibitem{MiltonNicorovici}
G.~W. Milton and N.-A.~P. Nicorovici.
\newblock On the cloaking effects associated with anomalous localized
  resonance.
\newblock {\em Proc. R. Soc. Lond. Ser. A Math. Phys. Eng. Sci.},
  462(2074):3027--3059, 2006.

\bibitem{Nguyen2012}
H.-M. Nguyen.
\newblock Approximate cloaking for the {H}elmholtz equation via transformation
  optics and consequences for perfect cloaking.
\newblock {\em Comm. Pure Appl. Math.}, 65(2):155--186, 2012.

\bibitem{Nguyen2013}
H.-M. Nguyen.
\newblock On a regularized scheme for approximate acoustic cloaking using
  transformation optics.
\newblock {\em SIAM J. Math. Anal.}, 45(5):3034--3049, 2013.

\bibitem{NV-siam}
H.-M. Nguyen and M.~S. Vogelius.
\newblock Approximate cloaking for the full wave equation via change of
  variables.
\newblock {\em SIAM J. Math. Anal.}, 44(3):1894--1924, 2012.

\bibitem{NV-arch}
H.-M. Nguyen and M.~S. Vogelius.
\newblock Full range scattering estimates and their application to cloaking.
\newblock {\em Arch. Ration. Mech. Anal.}, 203(3):769--807, 2012.

\bibitem{Pendry2006}
J.~B. Pendry, D.~Schurig, and D.~R. Smith.
\newblock Controlling electromagnetic fields.
\newblock {\em Science}, 312(5781):1780--1782, 2006.

\bibitem{Prasolov-Sossinsky-book}
V.~V. Prasolov and A.~B. Sossinsky.
\newblock {\em Knots, links, braids and 3-manifolds}, volume 154 of {\em
  Translations of Mathematical Monographs}.
\newblock American Mathematical Society, Providence, RI, 1997.
\newblock An introduction to the new invariants in low-dimensional topology,
  Translated from the Russian manuscript by Sossinsky [Sosinski\u{\i}].

\bibitem{Sanchez}
J.~Prat-Camps, C.~Navau, and A.~Sanchez.
\newblock Magnetic wormhole.
\newblock {\em Scientific Reports -- Nature}, 5:12488, 2015.

\bibitem{Saveliev-book}
N.~Saveliev.
\newblock {\em Lectures on the topology of 3-manifolds}.
\newblock De Gruyter Textbook. Walter de Gruyter \& Co., Berlin, revised
  edition, 2012.
\newblock An introduction to the Casson invariant.

\bibitem{Sequin}
C.~H. S\'{e}quin.
\newblock Tori story.
\newblock In R.~Sarhangi and C.~H. S\'{e}quin, editors, {\em Proceedings of
  Bridges 2011: Mathematics, Music, Art, Architecture, Culture}, pages
  121--130. Tessellations Publishing, 2011.

\bibitem{Uhlmann-ICM}
G.~Uhlmann.
\newblock Inverse boundary value problems for partial differential equations.
\newblock In {\em Proceedings of the {I}nternational {C}ongress of
  {M}athematicians ICM 1998, {V}ol. {III} ({B}erlin, 1998)}, pages 77--86,
  1998.

\bibitem{Uhlmann-review}
G.~Uhlmann.
\newblock Visibility and invisibility.
\newblock In {\em I{CIAM} 07---6th {I}nternational {C}ongress on {I}ndustrial
  and {A}pplied {M}athematics}, pages 381--408. Eur. Math. Soc., Z\"{u}rich,
  2009.

\bibitem{Uhlmann-Bull}
G.~Uhlmann.
\newblock Inverse problems: seeing the unseen.
\newblock {\em Bull. Math. Sci.}, 4(2):209--279, 2014.

\bibitem{Wallace}
A.~H. Wallace.
\newblock Modifications and cobounding manifolds.
\newblock {\em Canadian Journal of Mathematics}, 12:503--528, 1960.

\bibitem{Weder1}
R.~Weder.
\newblock The boundary conditions for point transformed electromagnetic
  invisibility cloaks.
\newblock {\em J. Phys. A}, 41(41):415401, 17, 2008.

\bibitem{Weder2}
R.~Weder.
\newblock A rigorous analysis of high-order electromagnetic invisibility
  cloaks.
\newblock {\em J. Phys. A}, 41(6):065207, 21, 2008.

\end{thebibliography}

\addcontentsline{toc}{section}{Bibliography}

\end{document}